\documentclass[a4paper,reqno]{amsart}

\usepackage{amssymb}
\usepackage{latexsym}
\usepackage{amsmath}
\usepackage{amsthm}
\usepackage{euscript}

\usepackage{pgfplots}
\usepackage{mathrsfs}
\usetikzlibrary{arrows}

\usepackage{mathtools}
\usepackage{changepage}
\usepackage{fge}
\usepackage{xcolor}
\usepackage[subnum]{cases}
\usepackage{tkz-euclide,amsmath}
\usetikzlibrary{quotes,angles}
\usetikzlibrary{intersections}
\pgfdeclarelayer{bg}
\pgfsetlayers{bg,main}
\usetikzlibrary{positioning}
\usetikzlibrary{calc}

\def\sa{{\mathfrak a}} 
\def\phi{\varphi}

\def\cD{{\mathcal D}}      
   \def\cH{{\mathcal H}}

\def\ee{{\mathrm e}}

\DeclareMathOperator\Real{{\text{\rm Re}}}

\DeclareMathOperator\sgn{{\text{\rm sgn}}} 
\DeclareMathOperator\dom{{\text{\rm dom\,}}} 

\DeclarePairedDelimiter{\abs}{\lvert}{\rvert}

\newcommand\ga[1]{\Gamma_{#1}}
\DeclareMathOperator{\om}{\Omega}
\newcommand\bnd{\partial \Omega}

\DeclareMathOperator{\R}{\mathbb{R}}
\DeclareMathOperator{\N}{\mathbb{N}}
\DeclareMathOperator{\C}{\mathbb{C}}
\def\dd{\mathrm{d}}

\newcommand\hmix{H^1_{0, \Gamma}(\Omega)}
\newcommand\hmixc{H^1_{0, \Gamma^c}(\Omega)}
\renewcommand\ga{\Gamma}
\newcommand\gac{{\Gamma^c}}
\newcommand\norm[1]{\| {#1} \|}
\newcommand\sign{\mathrm{sign}}
\newcommand\ldd{L^2(\Omega; \mathbb{C}^2)}

\renewcommand\setminus{\mathbin{\mathpalette\rsetminusaux\relax}}
\newcommand\rsetminusaux[2]{\mspace{-4mu}
  \raisebox{\rsmraise{#1}\depth}{\rotatebox[origin=c]{-20}{$#1\smallsetminus$}}
 \mspace{-4mu}
}
\newcommand\rsmraise[1]{%
  \ifx#1\displaystyle .8\else
    \ifx#1\textstyle .8\else
      \ifx#1\scriptstyle .6\else
        .45%
      \fi
    \fi
  \fi}
\usepackage{tikz}

\DeclareMathOperator{\diver}{div}

\makeatletter
\newcommand*{\rom}[1]{\expandafter\@slowromancap\romannumeral #1@}
\makeatother

\definecolor{darkspringgreen}{rgb}{0.09, 0.45, 0.27}

\definecolor{darktangerine}{rgb}{1.0, 0.66, 0.07}
\definecolor{amber}{rgb}{1.0, 0.75, 0.0}

\definecolor{purple}{rgb}{0.6, 0.4, 0.8}

\definecolor{amber(sae/ece)}{rgb}{1.0, 0.49, 0.0}

\newtheorem{theorem}{Theorem}[section]
\newtheorem*{thm*}{Theorem}
\newtheorem{proposition}[theorem]{Proposition}
\newtheorem{corollary}[theorem]{Corollary}
\newtheorem{lemma}[theorem]{Lemma}
\newtheorem{hypothesis}[theorem]{Hypothesis}

\theoremstyle{definition}

\theoremstyle{definition}

\newtheorem{example}[theorem]{Example}

\newtheorem{remark}[theorem]{Remark}

\numberwithin{equation}{section}
\numberwithin{figure}{section}

\title[]{On the first eigenvalue and eigenfunction of the Laplacian with mixed boundary conditions}

\author[N.~Aldeghi]{Nausica Aldeghi}
\author[J.~Rohleder]{Jonathan Rohleder}
\address{Matematiska institutionen \\ Stockholms universitet \\
106 91 Stockholm \\
Sweden}
\email{nausica.aldeghi@math.su.se, jonathan.rohleder@math.su.se}


\begin{document}

\begin{abstract}
We consider the eigenvalue problem for the Laplacian with mixed Dirichlet and Neumann boundary conditions. For a certain class of bounded, simply connected planar domains we prove monotonicity properties of the first eigenfunction. As a consequence, we establish a variant of the hot spots conjecture for mixed boundary conditions. Moreover, we obtain an inequality between the lowest eigenvalue of this mixed problem and the lowest eigenvalue of the corresponding dual problem where the Dirichlet and Neumann boundary conditions are interchanged. The proofs are based on a novel variational principle, which we establish.
\end{abstract}

\maketitle

\section{Introduction}

Eigenvalue problems for the Laplacian with mixed boundary conditions have been studied frequently in the last few years from various points of view \cite{ABIN20,CNT23p,FNO21,FNO22,SY20}. They also serve as tools in the investigation of other related topics such as nodal properties of eigenfunctions or the famous hot spots problem \cite{BH21,JM20}. 

In this article we are concerned with the lowest eigenvalue and the corresponding eigenfunction of the negative Laplacian with mixed Dirichlet--Neumann boundary conditions. For a bounded planar Lipschitz domain $\Omega$, consider the eigenvalue problem
\begin{align}\label{eq:EVproblem}
 \begin{cases}
 \hspace*{1.6mm} - \Delta \psi = \lambda \psi & \text{in}~\Omega, \\
 \nu \cdot \nabla \psi = 0 & \text{on}~\Gamma, \\
 \hspace{7.4mm} \psi = 0 & \text{on}~\Gamma^c;
 \end{cases}
\end{align}
here $\nu$ is the exterior unit normal vector field on the boundary, defined almost everywhere, and $\Gamma, \Gamma^c$ are relatively open, disjoint subsets of the boundary $\partial \Omega$ such that $\overline{\Gamma \cup \Gamma^c} = \partial \Omega$. The problem \eqref{eq:EVproblem} has a discrete sequence of non-negative eigenvalues unbounded from above; see Section \ref{sec:preliminaries} for more details. Its lowest eigenvalue is positive as soon as $\Gamma^c$ is non-empty, which we assume henceforth. The corresponding eigenfunction $\psi_1$ is unique up to multiplication by a scalar and can be chosen positive in $\Omega$. 

The first purpose of this article is connected to a question raised by Ba\~nuelos, Pang and Pascu \cite{BPP04}, see also \cite{BP04}: {\it What conditions must be imposed on $\Gamma$ and $\Gamma^c$ to ensure that the ground state eigenfunction $\psi_1$ attains its maximum only on the boundary?} 
This question is inspired by the hot spots conjecture about the minima and maxima of the first non-constant eigenfunction of the Laplacian with Neumann boundary conditions \cite{AB04,BB99,JN00,JM20} and is closely related to the long time behavior of solutions to the corresponding heat equation. Intuitively, one should expect the hot spots property for mixed boundary conditions to be true if the Dirichlet part is ``not too large'' in an appropriate sense taking into account the geometry of $\partial \Omega$.
In \cite[Corollary 1.2]{BPP04} it was shown that $\psi_1$ takes its maximum exclusively on $\partial \Omega$, i.e.\ on $\Gamma$, if $\Omega$ is convex, $\Gamma$ or $\Gamma^c$ is an arc of a circle,
 and the angles at which $\Gamma$ and $\Gamma^c$ meet are less than $\pi/2$; this extended an earlier result \cite{BP04} in which $\Gamma$ was assumed to be a line segment, see also \cite[Theorem 4.3]{BB99} and \cite{P02}. Very recently this question has attracted more attention: in the preprint \cite{H24p} Hatcher shows that the hot spots property for the problem \eqref{eq:EVproblem} is satisfied if $\Omega$ is a triangle and $\Gamma$ consists of either one or two sides of it, see also \cite{LY24p}, where more general semilinear problems are studied; moreover, Hatcher proves a result for certain domains for which $\Gamma$ is a line segment.

One of the main results of the present article is the following, where we denote by $Q_1, \dots, Q_4$ the first to fourth open quadrants in the plane and by $\ee_1 = (1, 0)^\top$ and $\ee_2 = (0, 1)^\top$ the standard basis vectors. 

\begin{theorem}\label{thm:intro1}
Assume that $\Omega$ is a bounded, simply connected Lipschitz domain with piecewise smooth boundary, that $\Gamma$ is connected, and that the unit normal field $\nu$ satisfies $\nu (x) \in \overline{Q_3}$ for almost all $x \in \Gamma^c$ and $\nu (x) \in \overline{Q_2} \cup \overline{Q_4}$ for almost all $x \in \Gamma$, see Figure \ref{fig:DomainsIntro}. Furthermore, assume that the interior angles of $\partial \Omega$ at which $\Gamma$ and $\Gamma^c$ meet are strictly less than $\pi/2$ and that $\partial \Omega$ has no inward-pointing corners. Then $\psi_1$ is strictly increasing in $\Omega$ in the directions of both $\ee_1$ and $\ee_2$. In particular, $\psi_1$ takes its maximum only on $\Gamma$.
\end{theorem}

We point out that the assumptions of Theorem \ref{thm:intro1} have the following implication: if there exist points $x_1, x_2 \in \Gamma$ with $\nu (x_1) \in Q_2$ and $\nu (x_2) \in Q_4$, then $\Gamma$ contains at least one corner, a point at which the normal vector jumps between $Q_2$ and $Q_4$. If, in addition, $\Gamma$ does not contain any axio-parallel piece, then this point is unique and it follows from Theorem \ref{thm:intro1} that this is the unique point in $\overline \Omega$ at which $\psi_1$ takes its maximum; see Figure~\ref{fig:DomainsIntro}. This is for instance the case if $\Omega$ is an acute triangle where $\Gamma$ consists of two sides and $\Gamma^c$ is the third. However, in general, under the assumptions of Theorem \ref{thm:intro1} it is also possible that $\Gamma$ contains a horizontal piece, in which case we cannot conclude uniqueness of the maximum point.
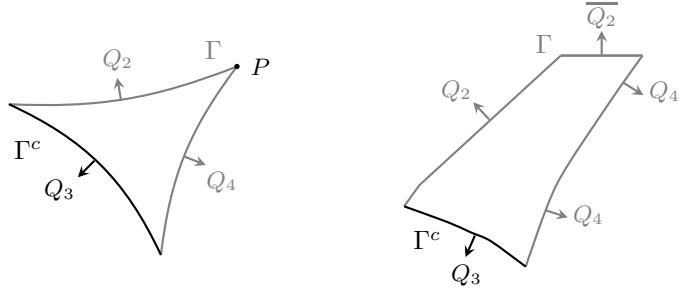
\begin{figure}[h]
\begin{minipage}[c][4cm][c]{0,3\textwidth}
\begin{tikzpicture}[scale=1]
\pgfsetlinewidth{0.8pt}
 \node[white] (A1) at (0.31,0.08) {};
 \node[white] (C1) at (2.06,-1.68) {};
 \draw[white] (A1.center) -- node[sloped,inner sep=0cm,above,pos=1.2, anchor=north east, minimum height=0.3cm,minimum width=1cm](F1){}(1,-0.6) -- (C1.center);
\path (F1.north east) edge[-stealth,shorten <=0.5pt] node[below left,pos=0.6] {\small $Q_3$} (F1.south east);
 \node[white] (W) at (1.83,-1.83) {};
 \node[white] (Z) at (2.74,0.46) {};
\draw[white] (W.center) -- node[sloped,inner sep=0cm,above,pos=1, anchor=north east, minimum height=0.3cm,minimum width=1cm](F4){}(2.285,-0.685) -- (Z.center);
\path[gray] (F4.north east) edge[-stealth,shorten <=0.5pt] node[gray,below right,pos=0.6] {\small $Q_4$} (F4.south east);
 \node[white] (J) at (0.5,-0.14) {};
 \node[white] (K) at (2.46,0.19) {};
\draw[white] (J.center) -- node[sloped,inner sep=0cm,above,pos=1, anchor=south west, minimum height=0.3cm,minimum width=1cm](F5){}(1.48,0.025) -- (K.center);
\path[gray] (F5.south west) edge[-stealth,shorten <=0.5pt] node[gray,above,pos=0.9] {\small $Q_2$} (F5.north west);
\path (0,0) coordinate (A) (3,0.5) coordinate (B) (2,-2) coordinate (C);
\path (0,0) coordinate (A) (3,0.5) coordinate (B) (2,-2) coordinate (C);
\draw[gray] (A) to[bend right=12] (B); 
\draw[gray] (B) to[bend right=15] (C); 
\draw[black] (A) to[bend left=20] (C); 
\draw(0.25,-0.61) node{$\gac$};
\draw(2.69,0.72) node[gray]{$\ga$};
\node[circle,fill=black,inner sep=0pt,minimum size=2pt,label=right:{\small {$P$}}] (C) at (3,0.5){};
\end{tikzpicture}
\end{minipage}
\hspace{1.1cm}
\begin{minipage}[c][4cm][c]{0,3\textwidth}
\begin{tikzpicture}[scale=0.4]
\pgfsetlinewidth{0.8pt}
\node[white] (A) at (2,4) {};
\node[white] (B) at (6,2) {};
\node[white] (C1) at (9.84,9) {};
\node[white] (D) at (7.15202,9) {};
\draw(6.6,9.4) node[gray]{$\ga$};
\node[white] (E) at (3.42633,3.48028) {};
\node[white] (F) at (4.17903,3.14355) {};
\node[white] (G) at (4.8525,2.84643) {};
\node[white] (H) at (7.011,4.71739) {};
\node[white] (I) at (8.74977,7.44076) {};
\draw(2.76,3.5) node[black][left, below]{$\gac$};
\draw[gray] (C1.center) -- node[sloped,inner sep=0cm, above,pos=.5, anchor=south west,
minimum height=0.3cm,minimum width=1cm](M){}(7.15,9) -- (D.center);
\draw[gray] (D.center) -- node[sloped,inner sep=0cm,above,pos=.5, anchor=south west,
minimum height=0.3cm,minimum width=1cm](N){}(2.5,4.7) -- (A.center);
\path[gray] (N.south west) edge[-stealth,shorten <=0.5pt] node[above left,pos=0.5] {\small $Q_2$} (N.north west);
\path[gray] (M.south west) edge[-stealth,shorten <=0.5pt] node[above,pos=0.7] {\small $\overline{Q_2}$} (M.north west);
\draw[white] (F.center) -- node[sloped,inner sep=0cm,above,pos=.5, anchor=north east, minimum height=0.3cm,minimum width=1cm](F1){}(4.51,3) -- (G.center);
\path (F1.north east) edge[-stealth,shorten <=0.5pt] node[below,pos=0.9] {\small $Q_3$} (F1.south east);
\draw[white] (B.center) -- node[sloped,inner sep=0cm,above,pos=.5, anchor=north east, minimum height=0.3cm,minimum width=1cm](F2){}(7.2,5.77) -- (H.center);
\path[gray] (F2.north east) edge[-stealth,shorten <=0.5pt] node[right,pos=0.8] {\small $Q_4$} (F2.south east);
\draw[white] (I.center) -- node[sloped,inner sep=0cm,above,pos=.5, anchor=north east, minimum height=0.3cm,minimum width=1cm](F3){}(9.6,8.8) -- (C1.center);
\path[gray] (F3.north east) edge[-stealth,shorten <=0.5pt] node[right,pos=0.8] {\small $Q_4$} (F3.south east);
\draw[gray] plot[smooth] coordinates {(B) (H) (I) (C1)};
\draw plot[smooth] coordinates {(A) (E) (F) (G) (B)};
\draw[gray] (C1.center) -- (D.center);
\end{tikzpicture}
\end{minipage}
\caption{Two domains for which Theorem \ref{thm:intro1} holds, rotated as required. The quadrants into which the exterior normal points are indicated. On the domain on the left, $\psi_1$ takes its unique maximum at $P$.}
\label{fig:DomainsIntro}
\end{figure}

It is worth mentioning that the conditions in Theorem \ref{thm:intro1} on the directions of the exterior normal vectors of $\Gamma$ and $\Gamma^c$ are natural and essentially necessary for $\psi_1$ to be strictly increasing in the directions of both $\ee_1$ and $\ee_2$. In fact, a transition between neighboring quadrants within $\Gamma$ would force one of $\partial_1 \psi_1$ and $\partial_2 \psi_1$ to change sign, due to the Neumann boundary condition; hence all the normals of $\Gamma$ must belong to two opposite quadrants. A similar reasoning applies to $\Gamma^c$ and its Dirichlet boundary condition, as the latter implies that the derivative of $\psi_1$ in tangential direction along $\Gamma^c$ vanishes. Moreover, the additional assumption that all normals of $\Gamma$ point into (the closure of) $Q_3$, instead of the union of $Q_1$ and $Q_3$, is suggested by the example of a rectangle with Dirichlet boundary conditions on two opposite sides and Neumann conditions on the remaining ones, where the first eigenfunction fails to be monotonic in both coordinate directions.

The second purpose of the present article is to prove an inequality for the first eigenvalue of the mixed problem. Associated with the eigenvalue problem \eqref{eq:EVproblem} we consider the dual problem
\begin{align}\label{eq:EVproblemDual}
 \begin{cases}
 \hspace*{1.6mm} - \Delta \phi = \lambda \phi & \text{in}~\Omega, \\
 \hspace{7.4mm} \phi = 0 & \text{on}~\Gamma, \\
 \nu \cdot \nabla \phi = 0 & \text{on}~\Gamma^c,
 \end{cases}
\end{align}
where the Dirichlet and Neumann boundary conditions have been interchanged. One can ask which of the problems \eqref{eq:EVproblem} and \eqref{eq:EVproblemDual} possesses the lower ground state eigenvalue. Simultaneously with Theorem \ref{thm:intro1} we prove the following result.

\begin{theorem}\label{thm:intro2}
Let the assumptions of Theorem \ref{thm:intro1} be satisfied. Moreover, denote by $\lambda_1^\Gamma$ the lowest eigenvalue of the problem \eqref{eq:EVproblemDual} and by $\lambda_1^{\Gamma^c}$ the lowest eigenvalue of \eqref{eq:EVproblem}. Then
\begin{align*}
 \lambda_1^{\Gamma^c} < \lambda_1^\Gamma
\end{align*}
holds.
\end{theorem}

This result extends and complements our result \cite[Theorem 3.1]{AR23}, where certain convex domains and straight $\Gamma^c$ were considered, see also \cite{A24p} for some results in higher dimensions. It is also related to recent results for triangles and other domains in \cite{R20,S16}. We point out that under the assumptions of Theorem \ref{thm:intro2} the length of $\Gamma^c$ is less than the length of $\Gamma$.

The methods we are using to prove both Theorems \ref{thm:intro1} and \ref{thm:intro2} are variational, yet non-standard. Instead of relying on the classical variational principle for the eigenvalues of the mixed Laplacian, we establish a novel variational principle, where the minimizers are gradients of eigenfunctions. A similar approach was used by the second author to approach the classical hot spots conjecture on the second Neumann Laplacian eigenfunction for a class of domains \cite{R21p} and extended to higher dimensions in \cite{KR24p}. A special case of what we prove here is the following theorem; cf.\ Theorems \ref{thm:VariationalPrincipleOneForm2} and \ref{thm:VariationalPrincipleOneForm3} below, where the results are stated precisely.

\begin{theorem}\label{thm:intro3}
Assume that $\Omega$ is a bounded, simply connected Lipschitz domain and that $\Gamma$ is connected. Let $\eta_1 = \min \{\lambda_1^\Gamma, \lambda_1^{\Gamma^c}\}$. Then
\begin{align}\label{eq:var1intro}
 \eta_1 = \min_{\substack{u = \binom{u_1}{u_2} \in \cH_\Gamma\\ u \neq 0}} \frac{\int_\Omega |\diver u|^2 + |\omega (u)|^2}{\int_\Omega |u|^2},
\end{align}
where $\omega (u) = \partial_1 u_2 - \partial_2 u_1$ and $\cH_\Gamma$ consists of all sufficiently regular vector fields $u = (u_1, u_2)^\top$ such that $u \cdot \nu = 0$ on $\Gamma$ and $u \cdot \tau = 0$ on $\Gamma^c$, where $\tau$ denotes the unit tangent vector along the boundary.

If, in addition, $\Omega$ has piecewise smooth boundary without inward-pointing corners and the interior angles of $\partial \Omega$ at which $\Gamma$ and $\Gamma^c$ meet are less than $\pi/2$, then
\begin{align}\label{eq:varintro2}
 \eta_1 = \min_{\substack{u = \binom{u_1}{u_2} \in \cH_\Gamma \\ u \neq 0}} \frac{\int_\Omega \left( |\nabla u_1|^2 + |\nabla u_2|^2 \right) - \int_{\partial \Omega} \kappa |u|^2}{\int_\Omega |u|^2},
\end{align}
where $\kappa$ denotes the signed curvature of $\partial \Omega$, with the sign chosen such that $\kappa \leq 0$ if $\Omega$ is convex.

In both variational principles, the minimizers can be expressed in terms of gradients of eigenfunctions of the problems \eqref{eq:EVproblem} and \eqref{eq:EVproblemDual} corresponding to their lowest eigenvalues.
\end{theorem}

For the understanding of the curvature term, it should be mentioned that the vector fields in $\cH_\Gamma$ vanish at all corners of $\partial \Omega$, in an appropriate weak sense, due to their boundary conditions. Hence the distributional nature of the curvature at corner points need not be taken into account. The first variational principle \eqref{eq:var1intro} applies to quite general situations, as it in particular does not require that the transition between $\Gamma$ and $\Gamma^c$ happens at corners. The second variational problem \eqref{eq:varintro2} naturally requires more regularity of $u$, which is ensured by the additional assumptions on $\partial \Omega$. It will turn out that under the assumptions of Theorem \ref{thm:intro1} one has $\eta_1 = \lambda_1^{\Gamma^c}$ and the corresponding minimizer $u$ of \eqref{eq:var1intro} respectively \eqref{eq:varintro2} equals $\nabla \psi_1$, the gradient of the first eigenfunction of \eqref{eq:EVproblem}.

This article is organized as follows. In Section \ref{sec:preliminaries} some necessary preliminaries are provided. Section \ref{sec:sec3} is devoted to the proof of the variational principles of Theorem \ref{thm:intro3}. Finally, in Section \ref{sec:main} these principles are applied to monotonicity properties of eigenfunctions and eigenvalue inequalities, that is, Theorems \ref{thm:intro1} and \ref{thm:intro2} are proven.

\section{Preliminaries}
\label{sec:preliminaries}

In this section we fix some notation and provide some preliminary material.

\subsection{Function spaces}

Let $\Omega \subset \mathbb{R}^2$ be a bounded, connected Lipschitz domain, see, e.g. \cite[Chapter 3]{M00}. By Rademacher's theorem, for almost all $x \in \bnd$ there exists a uniquely defined exterior unit normal vector $\nu(x) = (\nu_1(x), \nu_2(x))^\top$ and a unit tangent vector $\tau(x) = (\tau_1(x), \tau_2(x))^\top=(-\nu_2(x), \nu_1(x))^\top$ in the direction of positive orientation of the boundary. We denote by $H^s(\om)$, $s > 0$, the $L^2$-based Sobolev space of order $s$ on $\om$. On the boundary we will make use of the Sobolev space $H^{1/2}(\bnd)$ and its dual space $H^{-1/2}(\bnd)$. Recall that there exists a unique bounded, everywhere defined, surjective trace map from $H^1(\om)$ onto $H^{1/2}(\bnd)$ which continuously extends the mapping
\begin{equation*}
 C^\infty(\overline{\om}) \ni \phi \mapsto \phi |_{\bnd};
\end{equation*}
we write $\phi |_{\bnd}$ for the trace of $\phi \in H^1(\om)$ on $\partial \Omega$. Moreover, we denote by $(\cdot,\cdot)_{\bnd}$ the sesquilinear duality between $H^{-1/2}(\bnd)$ and $H^{1/2}(\bnd)$; in particular, 
\begin{equation*}
 (\phi, \psi)_{\bnd} = \int_{\bnd} \phi \, \overline{\psi}, \quad \psi \in H^{1/2}(\bnd),
\end{equation*}
if $\phi \in L^2(\partial \om) \subset H^{-1/2}(\bnd)$. 

Let us next come to spaces of vector fields. As usual, we denote by $L^2 (\Omega; \C^2)$ the space of square-integrable complex two-component vector fields on $\Omega$. For $u=(u_1,u_2)^{\top} \in L^2 (\Omega; \C^2)$ we define its {\em divergence} $\diver u$ and {\em vorticity} (or scalar curl) $\omega (u)$, respectively, by
\begin{equation*}
 \diver{u} \coloneqq \partial_1 u_1 + \partial_2 u_2 \quad \text{and} \quad \omega (u) \coloneqq \partial_1 u_2 - \partial_2 u_1,
\end{equation*}
where the derivatives are taken in the sense of distributions. In order to define the normal trace of a vector field in a weak sense, we consider the space
\begin{equation*}
 H (\diver; \Omega) \coloneqq \Biggl\{ u \in L^2(\om; \mathbb{C}^2) : \diver{u} \in L^2(\om) \Biggr\}
\end{equation*}
of square-integrable vector fields with square-integrable divergence. It is a Hilbert space when equipped with its natural norm defined by
\begin{equation*}
 \norm{u}^2_{H (\diver; \Omega)} = \int_{\om} \left(\abs{u_1}^2+\abs{u_2}^2+\abs{\diver{u}}^2 \right), \quad u \in H (\diver; \Omega).
\end{equation*}
The mapping
\begin{equation*}
 C^\infty(\overline{\om}; \mathbb{C}^2) \ni u \mapsto u|_{\bnd} \cdot \nu,
\end{equation*} 
where $\cdot$ denotes the standard inner product in $\mathbb{C}^2$, extends continuously to a bounded linear operator from $H (\diver; \Omega)$ into $H^{-1/2}(\partial \om)$; see \cite[Chapter I, Theorem 2.5]{GR}. We write $u|_{\bnd} \cdot \nu$ for the image of $u \in H (\diver; \Omega)$ under this mapping and call it the \textit{normal trace} of $u$. In particular, the integration by parts formula
\begin{align}
\label{eq:PI}
 \int_\Omega u \cdot \nabla \phi + \int_\Omega  \diver u \, \overline{\phi} = \big( u |_{\partial \Omega} \cdot  \nu, \phi \big)_{\partial \Omega}
\end{align}
holds for all $\phi \in H^1(\om)$ and $u \in H (\diver; \Omega)$. Analogously, once we equip the space
\begin{align*}
 H (\omega; \Omega) \coloneqq \Biggl\{ u \in L^2(\om; \mathbb{C}^2) : \omega (u) \in L^2(\om) \Biggr\}
\end{align*}
with the norm
\begin{align*}
 \norm{u}^2_{H (\omega; \Omega)} = \int_{\om} \left(\abs{u_1}^2+\abs{u_2}^2+\abs{\omega (u)}^2 \right), \quad u \in H (\omega; \Omega),
\end{align*}
the mapping 
\begin{equation*}
\quad C^\infty(\overline{\om}; \mathbb{C}^2) \ni u \mapsto u|_{\bnd} \cdot \tau
\end{equation*} 
extends continuously to a bounded linear operator from $H (\omega; \Omega)$ into $H^{-1/2}(\partial \om)$; we write $u|_{\bnd} \cdot \tau$ for the image of $u \in H (\omega; \Omega)$ under this mapping, and call it the \textit{tangential trace} of $u$. If we define $u^\perp := (- u_2, u_1)^\top$ and $\nabla^\perp \phi := (\nabla \phi)^\perp$, then
\begin{align}
\label{eq:PIomega}
 \int_\Omega u \cdot \nabla^\perp \phi + \int_\Omega \omega (u) \, \overline{\phi} = - \big(u^\perp |_{\bnd} \cdot \nu, \phi |_{\bnd} \big)_{\partial \Omega} = \big(u|_{\bnd} \cdot \tau, \phi |_{\bnd} \big)_{\partial \Omega}
\end{align}
holds for all $\phi \in H^1 (\Omega)$ and $u \in H (\omega; \Omega)$. The identity \eqref{eq:PIomega} can be deduced from \eqref{eq:PI} by noting that $\omega (u) = - \diver u^\perp$ and $\nu^\perp = \tau$.

It is well known that the space $L^2 (\Omega; \C^2)$ enjoys an orthogonal decomposition into the space $\nabla H^1 (\Omega)$ and the space of divergence-free vector fields $u$ with $u |_{\partial \Omega} \cdot \nu = 0$. In this paper we will use a variant of this decomposition which is tailor-made for the treatment of mixed boundary conditions. To this purpose, for any non-empty, relatively open set $\Sigma \subset \bnd$ we denote by $H_{0, \Sigma}^1 (\Omega)$ the Sobolev space
\begin{equation*}
 H_{0, \Sigma}^1 (\Omega) = \left\{ \phi \in H^1(\om): \phi |_{\Sigma}=0 \right\},
\end{equation*}
where $\phi |_\Sigma$ denotes the restriction of the trace $\phi |_{\bnd}$ to $\Sigma$. We say that a distribution $\psi \in H^{-1/2}(\bnd)$ vanishes on $\Sigma$, and write $\psi|_\Sigma = 0$, if
\begin{equation*}
 \left (\psi, \phi|_{\bnd} \right)_{\bnd} = 0
\end{equation*}
holds for all $\phi \in H_{0, \partial \Omega \setminus \overline{\Sigma}}^1 (\Omega)$. In this sense we write $u |_\Sigma \cdot \nu := (u |_{\partial \Omega} \cdot \nu) |_{\Sigma}$ for $u \in H (\diver; \Omega)$ and define analogously $u |_\Sigma \cdot \tau$ for $u \in H (\omega; \Omega)$. 

From now on let us fix two relatively open, non-empty sets $\Gamma, \Gamma^c \subset \partial \Omega$ such that
\begin{align}\label{eq:GammaDef}
 \Gamma \cap \Gamma^c = \emptyset \quad \text{and} \quad \overline{\Gamma \cup \Gamma^c} = \partial \Omega.
\end{align}
We define the space
\begin{equation*}
 H_c := \left\{ u \in L^2 (\Omega; \C^2) : \diver{u} = \omega(u) = 0, u |_{\Gamma} \cdot \nu = 0, u |_{{\Gamma^c}} \cdot \tau = 0 \right\};
\end{equation*}
note that it depends on $\ga$. Then $H_c$ is a closed subspace of $L^2(\Omega; \mathbb{C}^2)$ and the orthogonal decomposition 
\begin{equation}
\label{eq:Helmholtz}
 L^2(\Omega; \mathbb{C}^2) = \nabla \hmixc \oplus \nabla^{\perp} \hmix \oplus H_c
\end{equation}
holds. It is a variant of the classical Helmholtz decomposition and can, e.g., be deduced from \cite[Theorem 5.2]{BPS19} in the language of differential forms; for the convenience of the reader, we provide an elementary proof for our setting in Appendix \ref{sec:appendixHH}.

\subsection{The Laplacian with mixed boundary conditions}

Here we briefly recall the definition of the Laplacian with mixed boundary conditions. We assume again that $\Omega \subset \R^2$ is a bounded, connected Lipschitz domain and that $\Gamma$ and $\Gamma^c$ are non-empty, relatively open subsets of its boundary satisfying \eqref{eq:GammaDef}. For each $\psi \in H^1 (\Omega)$ such that $\Delta \psi \in L^2 (\Omega)$ distributionally, the gradient vector field $u := \nabla \psi$ belongs to $H (\diver; \Omega) \cap H (\omega; \Omega)$ as $\diver u = \Delta \psi \in L^2 (\Omega)$ and $\omega (u) = 0$. In particular, the \textit{normal derivative} and \textit{tangential derivative}
\begin{equation*}
 \partial_\nu \psi|_{\bnd} \coloneqq \nabla \psi |_{\bnd} \cdot \nu \quad \text{and} \quad \partial_\tau \psi|_{\bnd} \coloneqq \nabla \psi|_{\bnd} \cdot \tau
\end{equation*}
of $\psi$ are well-defined in $H^{-1/2}(\bnd)$ according to the previous section. If $\psi$ is more regular, say $\psi \in H^2(\om)$, then the normal and tangential derivatives can alternatively be defined via the ordinary trace map. Hence we can define the negative Laplacian $- \Delta_\Gamma$ subject to a Dirichlet boundary condition on $\ga$ and a Neumann boundary condition on $\gac$ as 
\begin{equation*}
 -\Delta_{\ga} \psi = -\Delta \psi, \quad \dom{(-\Delta_{\ga} )} = \left\{ \psi \in \hmix: \Delta \psi \in L^2(\om), \partial_\nu \psi |_{\gac} = 0 \right\}.
\end{equation*}
The operator $-\Delta_{\ga}$ corresponds to the non-negative, closed quadratic form
\begin{equation*}
 \hmix \ni \psi \mapsto \int_{\om} \abs{\nabla \psi}^2
\end{equation*}
in the sense of \cite[Chapter VI, Theorem 2.1]{Kato}; cf.\ Proposition \ref{prop:Kato} below. It is self-adjoint in $L^2(\om)$ and its spectrum consists of a discrete sequence of positive eigenvalues with finite multiplicities converging to $+\infty$. We point out that the operator $- \Delta_{\Gamma^c}$ with the dual boundary conditions, i.e.\ Dirichlet on $\Gamma^c$ and Neumann on $\Gamma$, is defined analogously.

Next we state a little lemma which in particular applies to functions in the domain of $- \Delta_\Gamma$. Intuitively, for a function being constant on a set $\Sigma \subset \bnd$, its tangential derivative must vanish on $\Sigma$. We give a short proof to make sure that this statement is true for the weakly defined tangential derivative and does not require any additional regularity of $\partial \Omega$.

\begin{lemma}
\label{lem:TangentialDerivative}
Assume that $\Sigma \subset \partial \Omega$ is relatively open and non-empty. Let $\phi \in H_{0, \Sigma}^1 (\Omega)$ such that $\Delta \phi \in L^2 (\Omega)$. Then 
\begin{equation*}
 \partial_\tau \phi |_\Sigma = 0
\end{equation*}
holds.
\end{lemma}

\begin{proof}
By the assumptions on $\phi$ we have $\nabla^\perp \phi \in \ldd$ and $\diver \nabla^\perp \phi = 0$. Thus, $\nabla^\perp \phi \in H (\diver; \Omega)$. Since $\phi \in H_{0, \Sigma}^1 (\Omega)$, there exists a sequence $(\phi_n)_n \subset  C^\infty(\overline \Omega) \cap H^1_{0, \Sigma} (\Omega)$ such that $\phi_n \to \phi$ in $H^1(\om)$, see, e.g., \cite[Theorem 1]{DZ06}. Then $\nabla^\perp \phi_n \to \nabla^\perp \phi$ in $H (\diver; \Omega)$ due to $\diver \nabla^\perp \phi_n = 0$ for all $n$ and $\diver \nabla^\perp \phi = 0$. In particular,
\begin{equation*}
 \partial_\tau \phi |_{\Sigma} = - \nabla^\perp \phi|_{\Sigma} \cdot \nu = - \lim_{n \to \infty} \nabla^\perp \phi_n|_{\Sigma} \cdot \nu = 0,
\end{equation*}
proving the lemma.
\end{proof}

Finally, let us emphasize that in general $\dom (- \Delta_\Gamma)$ is not contained in $H^2 (\Omega)$. However, such regularity holds under assumptions on the angles at the corners at which the transition between Dirichlet and Neumann boundary conditions takes place. 

\begin{proposition}
\label{prop:RegularityAngles}
If $\partial \Omega$ is piecewise smooth, contains no inward-pointing corners, and the interior angles at all points at which $\Gamma$ and $\Gamma^c$ meet are less or equal $\pi/2$, then
\begin{align*}
 \dom{(-\Delta_\Gamma)} \subset  H^2(\om)
\end{align*}
holds.
\end{proposition}

This regularity result for mixed boundary conditions is rather well known. For the case of a polygon the statement follows from \cite[Theorem 2.3.7]{G92} (see also the discussion below the theorem there). In the more general case, $H^2$-regularity of the functions in $\dom (- \Delta_\Gamma)$ away from the corners where the transition between Dirichlet and Neumann condition takes place is a standard result, see e.g.\ the survey \cite{KO83}. The corners at which $\Gamma$ and $\Gamma^c$ meet can be locally mapped conformally into a sector of the same angle and the result carries over from the polygonal case, see \cite{A81} or \cite[Section 5]{W70}.

Note that the assumptions of the previous proposition imply that the change between Neumann and Dirichlet boundary conditions happens at corners only, not in the interior of a smooth arc. This is not required in the definition of $- \Delta_\Gamma$ above.

\subsection{Sesquilinear forms and self-adjoint operators}

In this section we review briefly the theory of closed non-negative sesquilinear forms in Hilbert spaces and associated non-negative self-adjoint operators as to be found for instance in \cite[Chapter VI]{Kato} or \cite[Section X.3]{RS75}.

Let $\cH$ be a separable Hilbert space with inner product $(\cdot,\cdot)$ and corresponding norm $\norm{\cdot}$, and let $\sa$ be a sesquilinear form in $\cH$ with domain $\dom{\sa}$, that is, $\dom{\sa}$ is a linear subspace of $\cH$ and $\sa: \dom{\sa} \times \dom{\sa} \to \mathbb{C}$ is linear in the first entry and anti-linear in the second. A sesquilinear form $\sa$ is said to be \textit{symmetric} if the corresponding quadratic form $\sa[u] \coloneqq \sa[u,u]$ is real for all $u \in \dom{\sa}$, and \textit{non-negative} if $\sa[u] \ge 0$ for all $u \in \dom{\sa}$. In this case, 
\begin{equation*}
 (u,v)_{\sa} \coloneqq \sa[u,v] + (u,v)
\end{equation*}
defines an inner product on $\dom{\sa}$, and $\sa$ is said to be \textit{closed} if $\dom{\sa}$ equipped with $(\cdot,\cdot)_{\sa}$ and the associated norm $\|\cdot\|_\sa$ is a Hilbert space. The following correspondence between sesquilinear forms and self-adjoint operators is well known.

\begin{proposition}
\label{prop:Kato}
Assume that $\sa$ is a symmetric, non-negative, closed sesquilinear form whose domain $\dom \sa$ is dense in $\cH$. Then the following assertions hold.
\begin{enumerate}
\item There exists a unique self-adjoint, non-negative operator $A$ in $\cH$ with $\dom A \subset \dom \sa$ such that
\begin{align*}
 (A u, v) = \sa [u, v], \quad u \in \dom A, v \in \dom \sa.
\end{align*}
Moreover, $u \in \dom \sa$ belongs to $\dom A$ if and only if there exists $w \in \cH$ such that $\sa [u, v] = (w, v)$ holds for all $v \in \dom \sa$; in this case, $A u = w$.
 \item If, in addition, $\dom \sa$ is compactly embedded into $\cH$ then the spectrum of $A$ is purely discrete, i.e.\ it consists of isolated eigenvalues with finite multiplicities. Upon enumerating these eigenvalues non-decreasingly according to their multiplicities,
\begin{align*}
 \eta_1 \leq \eta_2 \leq \dots,
\end{align*}
the min-max principle
\begin{align*}
 \eta_j = \min_{\substack{F \subset \dom \sa \\ \dim F = j}} \;\; \max_{u \in F \setminus \{0\}} \frac{\sa [u]}{\|u\|^2}
\end{align*}
holds. In particular, the lowest eigenvalue $\eta_1$ of $A$ is given by
\begin{align}
\label{eq:minPrinciple}
 \eta_1 = \min_{\substack{u \in \dom \sa\\ u \neq 0}} \frac{\sa [u]}{\|u\|^2},
\end{align}
and $u \in \dom \sa$ with $u \neq 0$ is an eigenvector of $A$ corresponding to $\eta_1$ if and only if $u$ minimizes \eqref{eq:minPrinciple}.
\end{enumerate}
\end{proposition}

Sometimes it will be useful to work on a slightly smaller space than $\dom \sa$. Recall that a subspace $\cD \subset \dom \sa$ is called a \textit{core} of $\sa$ if $\cD$ is dense in $\dom \sa$ with respect to $\|\cdot\|_\sa$, and that $\dom A$ is always a core of $\sa$.

\begin{lemma}\label{lem:core}
Assume that $\sa$ is a symmetric, non-negative, densely defined, closed sesquilinear form in $\cH$ and that $A$ is the corresponding self-adjoint, non-negative operator of Proposition \ref{prop:Kato}. Assume that $\dom \sa$ is compactly embedded into $\cH$ and that $\cD \subset \dom A$ is a subspace such that $\cD$ contains all eigenvectors of $A$. Then $\cD$ is a core of $\sa$.
\end{lemma}

\begin{proof}
Let $(v_j)_j \subset \dom A$ be an orthonormal basis of $\cH$ such that $A v_j = \eta_j v_j$ holds for all $j \in \N$, where $\eta_j$ are the eigenvalues of $A$. Let $u \in \dom A$ and define, for each $N \in \N$,
\begin{align*}
 u_N := \sum_{j = 1}^N (u, v_j) v_j \in \cD.
\end{align*}
By the spectral theorem, $u_N$ converges to $u$ in $\cH$ and 
\begin{align*}
 A u_N = \sum_{j = 1}^N \eta_j (u, v_j) v_j
\end{align*}
converges to $A u$ in $\cH$. Hence,
\begin{align*}
 \|u - u_N\|_\sa^2 & = \|u - u_N\|^2 + \big( A (u - u_N), u - u_N \big) \\
 & \leq \|u - u_N\|^2 + \| A u - A u_N\| \|u - u_N\| \to 0
\end{align*}
as $N \to \infty$. Since $\dom A$ is a core of $\sa$, the claim follows.
\end{proof}

\section{Variational principles for mixed Dirichlet--Neumann Laplacian eigenvalues}
\label{sec:sec3}

In this section we establish two variants of a variational principle for the eigenvalues of the Laplacian with mixed boundary conditions. The first variant does not require any regularity, in addition to Lipschitz, on the boundary of $\Omega$. The second one, which will involve the curvature of the boundary, will need some smoothness.

\subsection{A div-curl variational principle for mixed boundary conditions}

In this subsection we generalize \cite[Section~3]{R23p}, where the case of Neumann and Dirichlet, instead of mixed, boundary conditions was considered. Throughout this section we impose the following assumption.

\begin{hypothesis}\label{hyp1}
The set $\Omega \subset \mathbb{R}^2$ is a bounded, connected Lipschitz domain and $\Gamma, \Gamma^c \subset \bnd$ are non-empty, relatively open subsets of the boundary such that 
\begin{align*}
 \Gamma \cap \Gamma^c = \emptyset \quad \text{and} \quad \overline{\Gamma \cup \Gamma^c} = \partial \Omega.
\end{align*}
\end{hypothesis}

We define a sesquilinear form $\sa$ in $L^2(\Omega; \mathbb{C}^2)$ by
\begin{equation}
\label{eq:aForm}
 \sa[u,v] = \int_{\Omega} \Big( \diver{u} \, \overline{\diver{v}} + \omega(u) \, \overline{\omega(v)} \Big), \quad u = \binom{u_1}{u_2}, v =  \binom{v_1}{v_2},
\end{equation}
with domain
\begin{align}\label{eq:aFormDomain}
\begin{split}
 \dom \sa & = \big\{ u \in L^2(\Omega; \C^2) : \diver{u}, \omega(u) \in L^2(\Omega), u |_{\Gamma} \cdot \nu = 0, \, u |_{\Gamma^c} \cdot \tau = 0 \big\}.
\end{split}
\end{align}
Note that the normal and tangential traces are well defined in the sense of Section~\ref{sec:preliminaries} as $\dom \sa \subset H (\diver; \Omega) \cap H (\omega; \Omega)$. Since 
\begin{equation*}
 \sa[u] = \int_{\Omega} \Big( \abs{\diver{u}}^2 + \abs{\omega(u)}^2 \Big) \ge 0,
\end{equation*}
one sees immediately that $\sa$ is symmetric and non-negative, so that 
\begin{equation*}
(u,v)_{\sa} : = \sa[u,v] + \int_\Omega u \cdot v
\end{equation*}
defines an inner product on the space $\cH_{\ga} \coloneqq \dom{\sa}$. 

The following proposition ensures that the form $\sa$ induces a self-adjoint operator $A$ in $L^2 (\Omega; \C^2)$.

\begin{proposition}\label{prop:FormToOpA}
Let Hypothesis \ref{hyp1} be satisfied. Then the sesquilinear form $\sa$ is closed, and $\dom \sa$ is dense in $L^2(\Omega; \mathbb{C}^2)$. In particular, there exists a self-adjoint operator $A$ in $L^2 (\Omega; \C^2)$ such that 
\begin{equation*}
 \int_\Omega A u \cdot v = \sa[u,v], \quad u \in \dom A, v \in \dom \sa.
\end{equation*}
A vector field $u \in \dom \sa$ belongs to $\dom A$ if and only if there exists $w \in L^2 (\Omega; \C^2)$ such that
\begin{equation*}
 \sa[u,v] = \int_\Omega w \cdot v
\end{equation*}
holds for all $v \in \dom \sa$, and in this case, $Au = w$. Moreover, the spectrum of $A$ is bounded below by 0.
\end{proposition}

\begin{proof}
If we prove the mentioned properties of $\sa$, the statements about the operator $A$ follow from Proposition~\ref{prop:Kato}. First, $\dom \sa$ is dense in $L^2 (\Omega; \C^2)$ as $C_0^\infty (\Omega; \C^2) \subset \dom \sa$ and $C_0^\infty (\Omega; \C^2)$ is dense in $L^2 (\Omega; \C^2)$. To see that $\sa$ is closed, we need to show that the space $\cH_{\ga}$ is complete. For this let $(u^n)_n$ be a Cauchy sequence in $\cH_{\ga}$, that is,
\begin{equation*}
 \int_{\Omega} \big( \abs{\diver{u^n}-\diver{u^m}}^2 + \abs{\omega(u^n)-\omega(u^m)}^2 + \abs{u^n-u^m}^2 \big) \to 0
\end{equation*}
as $n,m \to \infty$. Then there exist $f, g \in L^2(\om)$ and $u \in L^2(\om, \mathbb{C}^2)$ such that
\begin{equation*}
 \diver{u^n} \to f, \quad \omega(u^n) \to g, \quad u^n \to u
\end{equation*}
in $L^2(\om)$ or $L^2(\om, \mathbb{C}^2)$, respectively. For all $\psi \in C_0^\infty(\om)$ it follows by integration by parts \eqref{eq:PI} that
\begin{equation*}
 \int_{\Omega} \diver{u}\, \overline{\psi} = - \int_{\Omega} u \cdot \nabla \psi = - \lim_{n \to \infty} \int_{\Omega} u^n \cdot \nabla \psi = \lim_{n \to \infty} \int_{\Omega} \diver{u^n} \overline{\psi} = \int_{\Omega} f \, \overline{\psi},
\end{equation*}
that is, $\diver{u}=f \in L^2(\om)$. Analogously, applying \eqref{eq:PIomega} we get
\begin{equation*}
 \int_{\Omega} \omega(u)\, \overline{\psi} = - \int_{\Omega} u \cdot \nabla^{\perp} \psi = - \lim_{n \to \infty} \int_{\Omega} u^n \cdot \nabla^{\perp} \psi = \lim_{n \to \infty} \int_{\Omega} \omega(u^n) \overline{\psi} = \int_{\Omega} g\, \overline{\psi}
\end{equation*}
for all $\psi \in C_0^\infty (\Omega)$, from which $\omega(u)=g \in L^2(\om)$ follows. As for the boundary conditions, we have for all $\psi \in \hmixc$
\begin{align*}
 (u|_{\bnd} \cdot \nu, \psi |_{\partial \Omega})_{\bnd} & = \int_{\Omega} u \cdot \nabla \psi + \int_{\Omega} \diver{u} \, \overline{\psi} = \lim_{n \to \infty} \bigg( \int_{\Omega} u^n \cdot \nabla \psi + \int_{\Omega} \diver{u^n} \overline{\psi} \bigg) \\
 & = \lim_{n \to \infty} ({u^n}|_{\bnd} \cdot \nu, \psi |_{\partial \Omega})_{\bnd} = 0,
\end{align*}
implying that $u|_{\Gamma} \cdot \nu= 0$. Furthermore, for all $\psi \in \hmix$ we have
\begin{align*}
 (u|_{\bnd} \cdot \tau, \psi |_{\partial \Omega})_{\bnd} & = \int_{\Omega} u \cdot \nabla^{\perp} \psi + \int_{\Omega} \omega(u) \, \overline{\psi} = \lim_{n \to \infty} \bigg(\int_{\Omega} u^n \cdot \nabla^{\perp} \psi + \int_{\Omega} \omega(u^n) \overline{\psi} \bigg) \\
 & = \lim_{n \to \infty} (u^n |_{\bnd} \cdot \tau, \psi |_{\partial \Omega})_{\bnd} = 0 
\end{align*}
implying that $u|_{\Gamma^c} \cdot \tau = 0$. Therefore, $u \in \cH_\ga$, and $\norm{u^n-u}_\sa \to 0$ as $n \to \infty$. Thus, $\cH_{\ga}$ is complete and, thus, $\sa$ is closed.
\end{proof}

\begin{remark}\label{rem:action}
The self-adjoint operator $A$ defined in Proposition \ref{prop:FormToOpA} acts componentwise as the negative Laplacian, since for all $v \in C_0^\infty (\Omega; \C^2)$ and $u \in \dom A$,
\begin{align*}
 \int_\Omega A u \cdot v = \sa [u, v] = - \int_\Omega \left( \nabla \diver u \cdot v + \nabla^\perp \omega (u) \cdot v \right)
\end{align*}
by the integration by parts formulas \eqref{eq:PI} and \eqref{eq:PIomega}, and $\Delta u = \nabla \diver u + \nabla^\perp \omega (u)$.
\end{remark}

Next, we compute the spectrum of $A$. Recall that $-\Delta_\ga$ denotes the Laplacian on $\Omega$ with a Dirichlet boundary condition on $\Gamma$ and a Neumann boundary condition on $\Gamma^c$, and $-\Delta_\gac$ is the operator where the two boundary conditions are interchanged. Let us enumerate the eigenvalues of $- \Delta_{\Gamma^c}$, counted according to their multiplicities, as 
\begin{equation*}
0 < \lambda_1^{\Gamma^c} < \lambda_2^{\Gamma^c} \le \dots
\end{equation*}
and denote by $\psi_1, \psi_2, \dots$ an orthonormal basis of $L^2 (\Omega)$ such that
\begin{equation*}
 - \Delta_{\Gamma^c} \psi_k = \lambda_k^{\Gamma^c} \psi_k, \quad k = 1, 2, \dots.
\end{equation*}
Analogously, let the eigenvalues of $- \Delta_{\Gamma}$ be enumerated as
\begin{equation*}
 0 < \lambda_1^{\Gamma} < \lambda_2^{\Gamma} \le \dots
\end{equation*}
and let $\phi_1, \phi_2, \dots$ be an orthonormal basis of $L^2 (\Omega)$ such that
\begin{equation*}
- \Delta_{\Gamma} \phi_k = \lambda_k^{\Gamma} \phi_k, \quad k = 1, 2, \dots.
\end{equation*}

The following proposition relates the eigenvalues of the aforementioned mixed Laplacians to the spectrum of $A$. For this it is useful to recall the decomposition
\begin{equation*}
 L^2(\Omega; \mathbb{C}^2) = \nabla \hmixc \oplus \nabla^{\perp} \hmix \oplus H_c,
\end{equation*}
see \eqref{eq:Helmholtz} above. 

\begin{proposition}
\label{prop:TranslateEV}
Assume that Hypothesis \ref{hyp1} is satisfied. Let $A$ be the self-adjoint operator in $L^2 (\Omega; \C^2)$ associated with the sesquilinear form $\sa$ in Proposition \ref{prop:FormToOpA}. Then the following hold.
\begin{enumerate}
 \item For each $k \geq 1$, $\nabla \psi_k$ is non-trivial, belongs to $\dom A$, and satisfies $A \nabla \psi_k = \lambda_k^{\Gamma^c} \nabla  \psi_k$.  Moreover, the functions $\frac{1}{\sqrt{\lambda_k^{\Gamma^c}}} \nabla \psi_k$ form an orthonormal basis of $\nabla H^1_{0, \Gamma^c} (\Omega)$.
 \item For each $k \geq 1$, $\nabla^\perp \phi_k$ is non-trivial, belongs to $\dom A$, and satisfies $A \nabla^\perp \phi_k = \lambda_k^{\Gamma} \nabla^\perp \phi_k$.  Moreover, the functions $\frac{1}{\sqrt{\lambda_k^{\Gamma}}} \nabla^\perp \phi_k$ form an orthonormal basis of $\nabla^\perp H_{0, \Gamma} ^1 (\Omega)$.
 \item $\ker A = H_c$.
\end{enumerate}
In particular the spectrum of $A$ coincides with the union of the eigenvalues of $-\Delta_{\Gamma}$ and $-\Delta_{\Gamma^c}$, including multiplicities, and the eigenvalue $0$ with multiplicity $\dim H_c$.
\end{proposition}

\begin{proof}
(i) Let first $\psi$ be any eigenfunction of $-\Delta_{\Gamma^c}$ corresponding to an eigenvalue $\lambda$ and consider its gradient $\nabla \psi$, which is non-trivial as $\psi$ is non-constant, and belongs to $L^2(\Omega; \mathbb{C}^2)$. The identities
\begin{equation*}
 \diver{\nabla \psi} = \Delta \psi = -\lambda \psi 
\end{equation*}
and 
\begin{equation*}
 \omega(\nabla \psi) = \partial_1 \partial_2 \psi -\partial_2 \partial_1 \psi = 0
\end{equation*}
both hold on $\Omega$ in the distributional sense, so that $\diver{\nabla \psi},\omega(\nabla \psi) \in L^2(\om)$. The Neumann boundary condition to which $\psi$ is subject on $\Gamma$ yields
\begin{equation*}
 \nabla \psi |_\Gamma \cdot \nu = \partial_\nu \psi |_\Gamma = 0,
\end{equation*}
while the Dirichlet boundary condition on $\Gamma^c$ yields
\begin{equation*}
 \nabla \psi |_{\Gamma^c} \cdot \tau = \partial_\tau \psi |_{\Gamma^c} = 0;
\end{equation*}
cf.\ Lemma \ref{lem:TangentialDerivative}. Thus $\nabla \psi \in \dom \sa$ and, furthermore, for all $v \in \dom \sa$ we have
\begin{align*}
\begin{split}
 \sa[\nabla \psi, v] & = \int_{\om} \Delta \psi \, \overline{\diver{v}} = - \lambda \int_{\om} \psi \, \overline{\diver{v}} = \lambda \int_{\om} \nabla \psi \cdot v,
\end{split}
\end{align*}
where we used $v|_{\Gamma} \cdot \nu = 0$ and $\psi |_{\Gamma^c} = 0$. Proposition \ref{prop:FormToOpA} then yields $\nabla \psi \in \dom A$ and $A \nabla \psi = \lambda \nabla \psi$. Consequently, for each $k$, $\lambda_k^{\Gamma^c}$ is an eigenvalue of $A$ with corresponding eigenfunction $\nabla \psi_k$.

To show the orthonormal basis property, by Green's identity, for all $j, k$ we have
\begin{align*}
 \int_\Omega \nabla \psi_j \cdot \nabla \psi_k & = - \int_\Omega \Delta \psi_j \overline{\psi_k}  + \left( \partial_\nu \psi_j |_{\partial \Omega}, \psi_k \right)_{\partial \Omega} = \lambda_j^{\Gamma^c} \int_\Omega \psi_j \overline{\psi_k},
\end{align*}
where the boundary term vanishes as $\psi_k |_{\Gamma^c} = 0$ and $\partial_\nu \psi_j |_{\Gamma} = 0$. The right-hand side integral equals 0 if $k \neq j$ and $\lambda_j^{\Gamma^c}$ if $k=j$ as the eigenfunctions $\psi_k$ form an orthonormal basis in $L^2(\Omega)$. Therefore the functions $\frac{1}{\sqrt{\lambda_k^{\Gamma^c}}} \nabla \psi_k$ form an orthonormal system in $\nabla \hmixc$. To prove that this orthonormal system is an orthonormal basis, assume that $\psi \in H_{0, \Gamma^c}^1 (\Omega)$ is such that $\nabla \psi$ is orthogonal to all the $\nabla \psi_k$ in $L^2 (\Omega; \C^2)$. Then for all $k$ we have
\begin{align*}
 0 = \int_\Omega \nabla \psi_k \cdot \nabla \psi = - \int_\Omega \Delta \psi_k \overline{\psi} + \left( \partial_\nu \psi_k |_{\partial \Omega}, \overline{\psi} \right)_{\partial \Omega} = \lambda_k^{\Gamma^c} \int_\Omega \psi_k \overline{\psi},
\end{align*}
where the boundary term vanishes as $\psi |_{\Gamma^c} = 0$ and $\partial_\nu \psi_k |_\Gamma = 0$. Hence $\psi$ is orthogonal to all eigenfunctions $\psi_k$ of $-\Delta_{\Gamma^c}$, from which $\psi = 0$ follows, and the proof of (i) is complete. 

(ii) Let now $\phi$ be an eigenfunction of $-\Delta_{\Gamma}$ corresponding to an eigenvalue $\lambda$ and consider its perpendicular gradient $\nabla^{\perp} \phi \in L^2(\Omega; \mathbb{C}^2)$, which is non-trivial as $\phi$ is non-constant. Then 
\begin{equation*}
 \diver{\nabla^{\perp} \phi} = -\partial_1 \partial_2 \phi +\partial_2 \partial_1 \phi = 0
\end{equation*}
and
\begin{equation*}
 \omega(\nabla^{\perp} \phi) = \Delta \phi = - \lambda \phi
\end{equation*}
on $\Omega$, giving $\diver{\nabla^\perp \phi},$ $\omega(\nabla^\perp \phi) \in L^2(\om)$. Since $\phi$ is constantly zero along $\Gamma$ we have
\begin{equation*}
 \nabla^{\perp} \phi |_\Gamma \cdot \nu = - \nabla \phi |_\Gamma \cdot \tau = - \partial_\tau \phi |_{\Gamma} = 0,
\end{equation*}
where we used $\nu = -\tau^\perp$ and Lemma \ref{lem:TangentialDerivative}. On $\Gamma^c$, $\phi$ satisfies a Neumann boundary condition and thus
\begin{equation*}
\nabla^{\perp} \phi |_{\Gamma^c} \cdot {\tau} = \nabla \phi |_{\Gamma^c} \cdot {\nu} = \partial_\nu \phi |_{\Gamma^c} = 0.
\end{equation*}
We have shown $\nabla^{\perp} \phi \in \dom \sa$. Furthermore, for each $v \in \dom \sa$ we have
\begin{align*}
\begin{split}
 \sa \left[\nabla^{\perp} \phi, v \right] & = \int_{\om} \Delta \phi \, \overline{\omega(v)} = - \lambda \int_{\om} \phi \, \overline{\omega(v)} = \lambda \int_{\om} \nabla^{\perp} \phi \cdot v
\end{split}
\end{align*}
where we used $\phi |_\Gamma = 0$ and $v|_{\Gamma^c} \cdot \tau = 0$. Again this implies that $\nabla^{\perp} \phi \in \dom{A}$ and $A \nabla^{\perp} \phi = \lambda \nabla^{\perp} \phi$. Hence, for each $k$, $\nabla^\perp \phi_k$ is an eigenfunction of $A$ corresponding to the eigenvalue $\lambda_k^\Gamma$. An argument entirely analogous to the one in (i) proves that the functions $\frac{1}{\sqrt{\lambda^\Gamma_k}}\nabla^\perp \phi_k$ form an orthonormal basis in $\nabla^\perp \hmix$.

(iii) As the form $\sa$ is non-negative, each $u \in \dom \sa$ satisfies
\begin{equation*}
 \sa[u] = \int_{\Omega} \left( |\diver{u}|^2 + |\omega(u)|^2 \right) \geq 0,
\end{equation*}
with equality if and only if $\diver u = \omega (u) = 0$. Hence $\ker A = H_c$.

For the remaining assertion on the spectrum of $A$, observe that in (i)--(iii) we have found a family of eigenfunctions of $A$ which, according to the decomposition \eqref{eq:Helmholtz}, span all of $L^2 (\Omega; \C^2)$. As $A$ is self-adjoint, this implies that we have found the entire spectrum of $A$ and it is of the form stated in the proposition.
\end{proof}

As a consequence of the previous proposition we get the following variational principle for the eigenvalues of the mixed Laplacians $-\Delta_\ga$ and $-\Delta_\gac$. Recall that the space $\cH_\Gamma$ is given by
\begin{align*}
 \cH_\Gamma = \big\{ u \in L^2(\Omega; \C^2) : \diver{u}, \omega(u) \in L^2(\Omega), u |_{\Gamma} \cdot \nu = 0, \, u |_{\Gamma^c} \cdot \tau = 0 \big\}.
\end{align*}

\begin{theorem}
\label{thm:VariationalPrincipleOneForm}
Assume that Hypothesis \ref{hyp1} is satisfied. Denote by
\begin{align*}
 \eta_1 \leq \eta_2 \leq \dots
\end{align*}
the positive eigenvalues of $A$, i.e.\ the union of the eigenvalues of $- \Delta_{\ga}$ and $- \Delta_{\gac}$, counted according to their multiplicities; see Proposition \ref{prop:TranslateEV}. Then
\begin{align*}
\begin{split}
 \eta_j & = \min_{\substack{U \subset \cH_{\Gamma} \\ U \perp H_c \\ \dim U = j}} \max_{\substack{u \in U \\ u \neq 0}} \frac{\int_{\Omega} \left( \abs{\diver{u}}^2 + \abs{\omega(u)}^2 \right)}{\int_\Omega \abs{u}^2}.
\end{split}
\end{align*}
In particular, the first positive eigenvalue $\eta_1$ of $A$ is given by
\begin{align}
\label{eq:minMaxEv1OneForm1}
\begin{split}
 \eta_1 & = \min_{\substack{u \in \cH_{\Gamma} \\ u \perp H_c \\ u \neq 0}} \frac{\int_{\Omega} \left( \abs{\diver{u}}^2 + \abs{\omega(u)}^2 \right)}{\int_\Omega \abs{u}^2}.
\end{split}
\end{align}
Moreover, $u \in \cH_\Gamma$ with $u \perp H_c$ and $u \neq 0$ is a minimizer of \eqref{eq:minMaxEv1OneForm1} if and only if $u = \nabla \psi + \nabla^\perp \phi$ for some $\psi \in \ker (-\Delta_\gac - \eta_1)$ and $\phi \in \ker (-\Delta_\ga - \eta_1)$.
\end{theorem}

\begin{proof}
By Proposition \ref{prop:TranslateEV}, the positive eigenvalues $\eta_j$ of $A$ are precisely the eigenvalues of the operators $- \Delta_{\Gamma^c}$ and $- \Delta_\Gamma$, taking into account multiplicities, and 
\begin{align}\label{eq:eigenspace}
 \ker (A - \eta_j) = \nabla \ker (-\Delta_\gac - \eta_j) \oplus \nabla^\perp \ker (-\Delta_\ga - \eta_j)
\end{align}
holds for all $j$. Therefore all statements of the theorem follow from Proposition~\ref{prop:Kato} (ii) and the definition of the sesquilinear form $\sa$ in \eqref{eq:aForm}--\eqref{eq:aFormDomain}.
\end{proof}

Next, we specify assumptions under which the kernel of $A$ is empty and, hence, the spectrum of $A$ consists exclusively of the union of the spectra of the two mixed Laplacians.

\begin{proposition}
\label{prop:ThirdSpaceTrivial}
Assume that Hypothesis \ref{hyp1} is satisfied. 
If, in addition, $\Omega$ is simply connected and  $\Gamma$ is connected, then
\begin{equation*}
 \ker{A} = \{ 0 \}.
\end{equation*}
In particular,
\begin{equation}\label{eq:HelmholtzSimplified}
 L^2(\Omega; \mathbb{C}^2) = \nabla \hmixc \oplus \nabla^{\perp} \hmix.
\end{equation}
\end{proposition}

\begin{proof}
By Proposition \ref{prop:TranslateEV} (iii) we have $\ker{A} = H_c$. Let $u \in H_c$, i.e.\ $u \in L^2 (\Omega; \C^2)$, $\diver{u}=\omega(u)=0$ in $\Omega$, $u|_\Gamma \cdot \nu = 0$ and $u|_{\Gamma^c} \cdot \tau = 0$. Since $\om$ is simply connected, $\omega(u)=0$ implies that there exists $f \in H^1 (\Omega)$ such that $u = \nabla f$, see \cite[Chapter I, Theorem 2.9]{GR}. As
\begin{align*}
 \partial_\tau f |_{\Gamma^c} = u |_{\Gamma^c} \cdot \tau = 0
\end{align*}
and $\Gamma^c$ is connected due to the assumptions of the proposition, there exists $c \in \mathbb{C}$ such that $f |_{\Gamma^c} = c$ identically. However, we may replace $f$ by $f - c$ and therefore assume without loss of generality that $f \in H_{0, \Gamma^c}^1 (\Omega)$. Using $\diver u = 0$ in $\Omega$ and $u |_\Gamma \cdot \nu = 0$ we get
\begin{align*}
\begin{split}
 \int_{\om} \abs{u}^2 & = \int_{\om} u \cdot \nabla f = - \int_{\om} \diver{u} \, \overline{f} + \left( u |_{\partial \Omega} \cdot \nu, f |_{\partial \Omega} \right)_{\partial \Omega} = 0.
\end{split}
\end{align*}
Hence, $u=0$ identically on $\om$, that is, $H_c = \{0\}$. The equality \eqref{eq:HelmholtzSimplified} now follows immediately from \eqref{eq:Helmholtz}.
\end{proof}

\begin{remark}
Note that under the assumptions of Proposition \ref{prop:ThirdSpaceTrivial} the form $\sa$ is even positive, not only non-negative.
\end{remark}

We point out that, in contrast to the case of pure Neumann respectively pure Dirichlet boundary conditions considered in \cite{R23p}, $\Omega$ being simply connected is not sufficient for $\ker A$ to be trivial, as the following simple example shows.

\begin{example}
\label{ex:square}
Consider the case $\Omega = (0,1) \times (0,1)$ and choose $\Gamma$ to be the union of the two opposite vertical sides,
\begin{align*}
 \Gamma = \left\{(0, y) : y \in (0,1) \right\} \cup \left\{(1, y) : y \in (0,1) \right\},
\end{align*}
and $\Gamma^c$ to be the union of the two horizontal sides; in particular, neither $\Gamma$ nor $\Gamma^c$ is connected. Then both the normal vector on $\Gamma$ and the tangent vector on $\Gamma^c$ are constantly equal to $\pm (1, 0)^\top$ on each connected component of the respective part of $\partial \Omega$. In particular, the constant vector field $u = (0, 1)^\top$ satisfies $u |_\Gamma \cdot \nu = 0$ and $u |_{\Gamma^c} \cdot \tau = 0$, as well as $\diver u = \omega (u) = 0$ in $\Omega$. Hence we have found a non-trivial element $u \in H_c$. The question of the dimension of $H_c$ is studied further in 
Proposition \ref{prop:dimHC} in Appendix \ref{sec:appendixHH}.
\end{example}

Theorem \ref{thm:VariationalPrincipleOneForm} together with Proposition \ref{prop:ThirdSpaceTrivial} immediately yields the following.

\begin{theorem}
\label{thm:VariationalPrincipleOneForm2}
Assume that Hypothesis \ref{hyp1} is satisfied. In addition, assume that $\Omega$ is simply connected and that $\Gamma$ is connected. Denote by
\begin{align*}
 \eta_1 \leq \eta_2 \leq \dots
\end{align*}
the eigenvalues of $A$, i.e.\ the union of the eigenvalues of $- \Delta_{\ga}$ and $- \Delta_{\gac}$, counted according to their multiplicities. Then
\begin{align*}
\begin{split}
 \eta_j & = \min_{\substack{U \subset \cH_{\Gamma} \\ \dim U = j}} \max_{\substack{u \in U \\ u \neq 0}} \frac{\int_{\Omega} \left( \abs{\diver{u}}^2 + \abs{\omega(u)}^2 \right)}{\int_\Omega \abs{u}^2}.
\end{split}
\end{align*}
In particular, the first eigenvalue $\eta_1$ of $A$ is given by
\begin{align}
\label{eq:minMaxEv1OneForm}
\begin{split}
 \eta_1 & = \min_{\substack{u \in \cH_{\Gamma} \\  u \neq 0}} \frac{\int_{\Omega} \left( \abs{\diver{u}}^2 + \abs{\omega(u)}^2 \right)}{\int_\Omega \abs{u}^2}.
\end{split}
\end{align}
Moreover, $u \in \cH_\Gamma$ with $u \neq 0$ is a minimizer of \eqref{eq:minMaxEv1OneForm} if and only if $u = \nabla \psi + \nabla^\perp \phi$ for some $\psi \in \ker (-\Delta_\gac - \eta_1)$ and $\phi \in \ker (-\Delta_\ga - \eta_1)$.
\end{theorem}

\subsection{A curvature variational principle for mixed boundary conditions}

In this subsection we prove an alternative representation of the sesquilinear form $\sa$ and, thus, of the variational principle of Theorem \ref{thm:VariationalPrincipleOneForm2}. In order to do this we make the following assumption.

\begin{hypothesis}\label{hyp2}
The set $\Omega \subset \mathbb{R}^2$ is a bounded, simply connected Lipschitz domain whose boundary consists of finitely many $C^\infty$-smooth arcs. Moreover, $\Gamma, \Gamma^c \subset \bnd$ are non-empty, relatively open, connected subsets of the boundary such that 
\begin{align*}
 \Gamma \cap \Gamma^c = \emptyset \quad \text{and} \quad \overline{\Gamma \cup \Gamma^c} = \partial \Omega.
\end{align*}
The interior angles at the two points where $\Gamma$ and $\Gamma^c$ meet are strictly below $\pi/2$, and $\partial \Omega$ has no inward-pointing corners.
\end{hypothesis}

We point out that Proposition \ref{prop:ThirdSpaceTrivial} and Theorem \ref{thm:VariationalPrincipleOneForm2} apply under Hypothesis \ref{hyp2}. 

Let us first establish additional regularity of the space $\dom \sa = \cH_\Gamma$ under the above hypothesis.

\begin{lemma}\label{lem:formDomain}
Assume that Hypothesis \ref{hyp2} is satisfied. Then
\begin{align*}
 \cH_\Gamma = \left\{ u \in H^1 (\Omega; \C^2) : u |_\Gamma \cdot \nu = 0, u |_{\Gamma^c} \cdot \tau = 0 \right\}
\end{align*}
holds.
\end{lemma}

\begin{proof}
It is clear that each $u \in H^1 (\Omega; \C^2)$ which satisfies the desired boundary conditions belongs to $\cH_\Gamma$. Conversely, each $u \in \cH_\Gamma$ may, according to Proposition \ref{prop:ThirdSpaceTrivial}, be written as $u = \nabla \psi + \nabla^\perp \phi$ with $\psi \in H_{0, \Gamma^c}^1 (\Omega)$ and $\phi \in H_{0, \Gamma}^1 (\Omega)$. Furthermore,
\begin{align*}
 \Delta \psi = \diver u \in L^2 (\Omega) \quad \text{and} \quad \Delta \phi = \omega (u) \in L^2 (\Omega),
\end{align*}
and on the boundary we have 
\begin{align*}
 \partial_\nu \psi |_\Gamma =  u |_\Gamma \cdot \nu - \nabla^\perp \phi |_\Gamma \cdot \nu = 0
\end{align*}
as well as
\begin{align*}
 \partial_\nu \phi |_{\Gamma^c} = \nabla^\perp \phi |_{\Gamma^c} \cdot \tau = u |_{\Gamma^c} \cdot \tau - \nabla \psi |_{\Gamma^c} \cdot \tau = 0,
\end{align*}
see Lemma \ref{lem:TangentialDerivative}. Hence, $\psi \in \dom (- \Delta_{\Gamma^c})$ and $\phi \in \dom (- \Delta_\Gamma)$ and, in particular, $\psi, \phi \in H^2 (\Omega)$ by Proposition \ref{prop:RegularityAngles}. It follows that $u = \nabla \psi + \nabla^\perp \phi \in H^1 (\Omega; \C^2)$, which completes the proof.
\end{proof}

Now we can state an alternative expression for the sesquilinear form $\sa$. To this end we denote by $\kappa$ the signed curvature of $\bnd$ with respect to the exterior unit normal $\nu$, defined at each point of $\bnd$ except at the corners, see for instance \cite[Exercise 8, Section 2.3]{O06} or \cite[Section 2.2]{PR10}. That is, the sign is chosen such that $\kappa \leq 0$ almost everywhere if $\Omega$ is convex. It can be expressed as 
\begin{equation*}
 \kappa = \nu \cdot \tau'
\end{equation*}
where the derivative $\tau'$ of the unit tangent vector field $\tau$ is to be understood piecewise via an arclength parametrization of the boundary in positive direction.

\begin{proposition}\label{prop:equalForms}
Assume that Hypothesis \ref{hyp2} is satisfied. Then
\begin{align}\label{eq:equalForms}
 \sa [u, v] = \int_\Omega \big( \nabla u_1 \cdot \nabla v_1+ \nabla u_2 \cdot \nabla v_2 \big)- \int_{\partial \Omega} \kappa \, u \cdot v
\end{align}
holds for all $u = (u_1, u_2)^\top, v = (v_1, v_2)^\top \in \cH_\Gamma$.
\end{proposition}

\begin{proof}
The proof of this proposition will be carried out in two steps.

{\bf Step~1.} Let us denote by $\cD$ the space of all $u \in \cH_\Gamma$ such that
\begin{enumerate}
 \item[(a)]  $u \in C^\infty (\Omega; \C^2)$ and for each open set $U \subset \Omega$ whose closure does not contain any corners of $\partial \Omega$, $u \in C^\infty (\overline{U}; \C^2)$, and
 \item[(b)] $\Delta u \in L^2 (\Omega; \C^2)$.
\end{enumerate}
In this step we prove that \eqref{eq:equalForms} holds whenever $u \in \cD$ and $v \in \cH_\Gamma$. 

To this end, let $u \in \cD$. Due to condition (b) we can apply the integration by parts formulas \eqref{eq:PI} and \eqref{eq:PIomega} to obtain, for all $v \in \cH_\Gamma$,
\begin{align*}
 \sa[u,v] & = \int_{\Omega} \big( \diver{u} \, \overline{\diver{v}} + \omega(u) \, \overline{\omega(v)} \big) \\
 & = - \int_{\Omega} \nabla \diver{u} \cdot v + \left(\diver{u} |_{\partial \Omega}, v|_{\bnd} \cdot \nu \right)_{\bnd} \\
 & \quad - \int_{\Omega} \nabla^{\perp} \omega(u) \cdot v + \left(\omega(u) |_{\partial \Omega}, v|_{\bnd} \cdot \tau \right)_{\bnd} \\
 & = - \int_{\Omega} \Delta u_1 \overline{v_1} - \int_{\Omega} \Delta u_2 \overline{v_2} +\left(\diver{u} |_{\partial \Omega}, v|_{\bnd} \cdot \nu \right)_{\bnd} + \left(\omega(u) |_{\partial \Omega}, v|_{\bnd} \cdot \tau \right)_{\bnd};
\end{align*}
the last equality relies on the identity $\Delta u = \nabla \diver u + \nabla^\perp \omega (u)$. An application of Green's identity then yields
\begin{align}
\label{eq:partialInt}
\begin{split}
 \sa [u,v] & = \int_{\om} \nabla u_1 \cdot \nabla v_1 +  \int_{\om} \nabla u_2 \cdot \nabla v_2  - \left(\partial_\nu u_1 |_{\partial \Omega}, v_1 \right)_{\bnd} - \left(\partial_\nu u_2 |_{\partial \Omega}, v_2 \right)_{\bnd} \\
 & \quad + \left(\diver{u} |_{\partial \Omega}, v|_{\bnd} \cdot \nu \right)_{\bnd} + \left(\omega(u) |_{\partial \Omega}, v|_{\bnd} \cdot \tau \right)_{\bnd}.
\end{split}
\end{align}
It remains to compute the boundary terms. To avoid technicalities at the corners, we perform the computations first for more regular $v$ and extend the result to $v \in \cH_\Gamma$ by approximation afterwards.

Let $\Gamma_1, \dots, \Gamma_N$ be the smooth arcs which constitute $\partial \Omega$ and let $v \in C^\infty (\overline \Omega; \C^2)$ be such that the support of $v$ does not contain any corner of $\partial \Omega$; note that for the moment we do not require any boundary conditions on $v$. In this case, due to the assumption (a) on $u$ we may write
\begin{align}\label{eq:computeBdr}
\begin{split}
 \left(\partial_\nu u_1 |_{\partial \Omega}, v_1 \right)_{\bnd} + \left(\partial_\nu u_2 |_{\partial \Omega}, v_2 \right)_{\bnd} = \sum_{l = 1}^N \int_{\Gamma_l} \binom{\nabla u_1 \cdot \nu}{\nabla u_2 \cdot \nu} \cdot v.
\end{split}
\end{align}
Let us fix some $l \in \{1, \dots, N\}$ and let $r : [0, L] \to \R^2$ be a $C^\infty$-smooth arc length parametrization of $\Gamma_l$, taken in positive direction. Then at each point $x = r (s)$ with $s \in (0, L)$,
\begin{align}
\label{eq:curveProperties}
 \tau (x) = r' (s), \quad \nu (x) = - r' (s)^\perp, \quad \text{and} \quad \kappa (x) = \nu (x) \cdot r'' (s) = \tau (x) \cdot r'' (s)^\perp.
\end{align}
Therefore for $x = r (s)$, $s \in (0, L)$,  we have 
\begin{align*}
 \frac{\dd}{\dd s} (u_j (r (s))) = (\nabla u_j) (r (s)) \cdot r' (s) = \partial_1 u_j (x) \tau_1 (x) + \partial_2 u_j (x) \tau_2 (x), \quad j = 1, 2,
\end{align*}
and thus
\begin{align}\label{eq:e1}
 \binom{\partial_1 u (x) \cdot \tau (x)}{\partial_2 u (x) \cdot \tau (x)} \cdot \nu (x) 
 & = \frac{\dd u (r (s))}{\dd s} \cdot \nu (x) + \partial_1 u_2 (x) - \partial_2 u_1 (x).
\end{align}
Analogously,
\begin{align}\label{eq:e2}
 \binom{\partial_1 u (x) \cdot \nu (x)}{\partial_2 u (x) \cdot \nu (x)} \cdot \nu (x) = - \frac{\dd u (r (s))}{\dd s} \cdot \tau(x) + \partial_1 u_1 (x) + \partial_2 u_2 (x).
\end{align}
Now we distinguish two cases. If $\Gamma_l \subset \Gamma$ then $u |_{\Gamma_l} \cdot \nu = 0$ and thus, using \eqref{eq:curveProperties},
\begin{align*}
 \frac{\dd u (r (s))}{\dd s} \cdot \nu (x) & = \frac{\dd}{\dd s} (u (r (s))  \cdot \nu (r (s))) - u (x) \cdot \frac{\dd \nu (r (s))}{\dd s} \\
 & = (u (x) \cdot \tau (x)) \tau (x) \cdot r'' (s)^\perp = \kappa (x) u (x) \cdot \tau (x).
\end{align*}
Hence, if we decompose $v$ into its tangential and normal components on $\Gamma_l$, \eqref{eq:e1} yields
\begin{align*}
 \int_{\Gamma_l} \binom{\nabla u_1 \cdot \nu}{\nabla u_2 \cdot \nu} \cdot v & = \int_{\Gamma_l} (\tau \cdot v) \binom{\partial_1 u \cdot \tau}{\partial_2 u \cdot \tau} \cdot \nu + \int_{\Gamma_l} \binom{\nabla u_1 \cdot \nu}{\nabla u_2 \cdot \nu} \cdot (v \cdot \nu) \nu \\
 & = \int_{\Gamma_l} \big( \kappa \, u \cdot v + \omega (u) \, \tau \cdot v \big) + \int_{\Gamma_l} \binom{\nabla u_1 \cdot \nu}{\nabla u_2 \cdot \nu} \cdot (v \cdot \nu) \nu,
\end{align*}
where we used $u \cdot v = (u \cdot \tau) (\tau \cdot v)$. On the other hand, if $\Gamma_l \subset \Gamma^c$, then $u |_{\Gamma_l} \cdot \tau = 0$ and we obtain from \eqref{eq:curveProperties}
\begin{align*}
 - \frac{\dd u (r (s))}{\dd s} \cdot \tau(x) = u (x) \cdot \frac{\dd \tau (r (s))}{\dd s} = \kappa (x) u (x) \cdot \nu (x)
\end{align*}
and hence, from \eqref{eq:e2},
\begin{align*}
 \int_{\Gamma_l} \binom{\nabla u_1 \cdot \nu}{\nabla u_2 \cdot \nu} \cdot v & = \int_{\Gamma_l} (\nu \cdot v) \binom{\partial_1 u \cdot \nu}{\partial_2 u \cdot \nu} \cdot \nu + \int_{\Gamma_l} \binom{\nabla u_1 \cdot \nu}{\nabla u_2 \cdot \nu} \cdot (v \cdot \tau) \tau \\
 & = \int_{\Gamma_l} \big( \kappa \, u \cdot v + \diver u \, \nu \cdot v \big) + \int_{\Gamma_l} \binom{\nabla u_1 \cdot \nu}{\nabla u_2 \cdot \nu} \cdot (v \cdot \tau) \tau.
\end{align*}
Therefore, taking into account \eqref{eq:computeBdr}, the boundary terms in \eqref{eq:partialInt} sum up to
\begin{align}\label{eq:BT}
 - \int_{\partial \Omega} \kappa \, u \cdot v - \sum_{\Gamma_l \subset \Gamma} \int_{\Gamma_l} \binom{\nabla u_1 \cdot \nu}{\nabla u_2 \cdot \nu} \cdot (v \cdot \nu) \nu - \sum_{\Gamma_l \subset \Gamma^c} \int_{\Gamma_l} \binom{\nabla u_1 \cdot \nu}{\nabla u_2 \cdot \nu} \cdot (v \cdot \tau) \tau.
\end{align}

Let now $v \in \cH_\Gamma$. By \cite[Chapter 8, Corollary 6.4]{EE18} there exists a sequence $(v^n)_n$ of functions in $C^\infty (\overline{\Omega}; \C^2)$ whose supports do not contain the corners of $\partial \Omega$ such that $v^n \to v$ in $H^1 (\Omega; \C^2)$. Then $v^n |_{\Gamma_l} \to v |_{\Gamma_l}$ in $H^{1/2} (\Gamma_l)$ for each $l$ by the continuity of the trace operator from $H^1 (\Omega)$ to $H^{1/2} (\partial \Omega)$. In particular, $v^n |_{\Gamma_l} \cdot \nu \to v |_{\Gamma_l} \cdot \nu$ and $v^n |_{\Gamma_l} \cdot \tau \to v |_{\Gamma_l} \cdot \tau$ in $H^{1/2} (\Gamma_l)$. Hence using the boundary conditions $v |_\Gamma \cdot \nu = 0$ and $v |_{\Gamma^c} \cdot \tau = 0$, it follows from \eqref{eq:BT} and \eqref{eq:partialInt} that
\begin{align}\label{eq:fast}
 \sa [u, v] = \int_{\om} \nabla u_1 \cdot \nabla v_1 +  \int_{\om} \nabla u_2 \cdot \nabla v_2 - \int_{\partial \Omega} \kappa \, u \cdot v
\end{align}
is true for all $u \in \cD$, $v \in \cH_\Gamma$.

{\bf Step~2.} In this step we prove the assertion of the proposition; that is, we extend \eqref{eq:fast} to all $u \in \cH_\Gamma$. First of all, note that, for all $u \in \cH_\Gamma$,
\begin{align*}
 \sa [u] & = \int_\Omega \left( |\diver u|^2 + |\omega (u)|^2 \right) = \int_\Omega \left( |\nabla u_1|^2 + |\nabla u_2|^2 - 2 \Real \nabla u_1 \cdot \nabla^\perp u_2 \right) \\
 & \leq 2 \int_\Omega \left( |\nabla u_1|^2 + |\nabla u_2|^2 \right) \leq 2 \|u\|_{H^1 (\Omega; \C^2)}^2.
\end{align*}
Hence there exists a constant $C > 0$ such that 
\begin{align*}
 \|u\|_\sa^2 = \sa [u] + \int_\Omega |u|^2 \leq C \|u\|_{H^1 (\Omega; \C^2)}^2, \quad u \in \cH_\Gamma.
\end{align*}
Since $\cH_\Gamma$ is complete with respect to both norms, the open mapping theorem implies that the two norms are equivalent on $\cH_\Gamma$. 

Consider the space $\cD \subset \cH_\Gamma$ defined in Step~1. By elliptic regularity, see, e.g., \cite[Theorem 4.18 (ii)]{M00},
\begin{align*}
 \ker (A - \eta_j) = \nabla \ker (- \Delta_{\Gamma^c} - \eta_j) \oplus \nabla^\perp \ker (- \Delta_\Gamma - \eta_j) \subset \cD
\end{align*}
holds for each $j$; cf.\ Proposition \ref{prop:TranslateEV}. It follows from Lemma \ref{lem:core} that $\cD$ is a core of $\sa$, i.e., $\cD$ is dense in $\cH_\Gamma$ with respect to the norm $\| \cdot \|_\sa$. By the equivalence of the norms, $\cD$ is also dense in $\cH_\Gamma$ with respect to the norm $\|\cdot\|_{H^1 (\Omega; \C^2)}$. Using the continuity of the trace map $H^1 (\Omega; \C^2) \to H^{1/2} (\partial \Omega; \C^2)$ it follows that \eqref{eq:fast} extends to all $u, v \in \cH_\Gamma$. This completes the proof.
\end{proof}

As an immediate consequence of Proposition \ref{prop:equalForms} and Theorem \ref{thm:VariationalPrincipleOneForm2} we obtain the following variational principle. Recall that by Lemma~\ref{lem:formDomain}
\begin{align*}
 \cH_\Gamma = \left\{ u \in H^1 (\Omega; \C^2) : u |_\Gamma \cdot \nu = 0, u |_{\Gamma^c} \cdot \tau = 0 \right\}
\end{align*}
holds as soon as Hypothesis \ref{hyp2} is satisfied.

\begin{theorem}
\label{thm:VariationalPrincipleOneForm3}
Assume that Hypothesis \ref{hyp2} is satisfied. Denote by
\begin{align*}
 \eta_1 \leq \eta_2 \leq \dots
\end{align*}
the eigenvalues of $A$, i.e.\ the union of the eigenvalues of $- \Delta_{\ga}$ and $- \Delta_{\gac}$, counted according to their multiplicities. Then
\begin{align*}
\begin{split}
 \eta_j & = \min_{\substack{U \subset \cH_{\Gamma} \\ \dim U = j}} \max_{\substack{u \in U \\ u \neq 0}} \frac{\int_{\Omega} \left( |\nabla u_1|^2 + |\nabla u_2|^2 \right) - \int_{\partial \Omega} \kappa \, |u|^2}{\int_\Omega \abs{u}^2}.
\end{split}
\end{align*}
In particular, the first eigenvalue $\eta_1$ of $A$ is given by
\begin{align}
\label{eq:minMaxEv1OneForm3}
\begin{split}
 \eta_1 & = \min_{\substack{u \in \cH_{\Gamma} \\  u \neq 0}} \frac{\int_{\Omega} \left( |\nabla u_1|^2 + |\nabla u_2|^2 \right) - \int_{\partial \Omega} \kappa \, |u|^2}{\int_\Omega \abs{u}^2}.
\end{split}
\end{align}
Moreover, $u \in \cH_\Gamma$ with $u \neq 0$ is a minimizer of \eqref{eq:minMaxEv1OneForm3} if and only if $u = \nabla \psi + \nabla^\perp \phi$ for some $\psi \in \ker (-\Delta_\gac - \eta_1)$ and $\phi \in \ker (-\Delta_\ga - \eta_1)$.
\end{theorem}

\section{Eigenvalue inequalities and monotonicity of eigenfunctions}\label{sec:main}

In this section we apply the variational principle of Theorem \ref{thm:VariationalPrincipleOneForm3} to obtain an inequality between the lowest eigenvalues of two mixed Laplacian eigenvalue problems with boundary conditions dual to each other. At the same time we establish monotonicity of the corresponding eigenfunction. In the following we denote by $Q_1, \dots, Q_4$ the open first to fourth quadrants in $\R^2$. The proof of the following result is inspired by the proof of \cite[Theorem~1.2]{R21p} and by methods from the proof of Courant's nodal domain theorem. The following theorem comprises Theorems \ref{thm:intro1} and \ref{thm:intro2} in the introduction.

\begin{theorem}
\label{thm:MainResult}
Assume that Hypothesis \ref{hyp2} is satisfied. Assume in addition that $\om$ can be rotated such that $\nu(x) \in \overline{Q_3}$ for almost all $x \in \gac$ and $\nu(x) \in \overline{Q_2} \cup \overline{Q_4}$ for almost all $x \in \ga$; cf.\ Figure~\ref{fig:DomainsIntro}. Then
\begin{equation}
\label{eq:mainEVin}
 \lambda_1^{\Gamma^c} < \lambda_1^\Gamma
\end{equation}
holds. Furthermore, the eigenfunction of $- \Delta_{\Gamma^c}$ corresponding to $\lambda_1^{\Gamma^c}$, chosen positive, is strictly increasing in the directions of both $\ee_1 = (1, 0)^\top$ and $\ee_2 = (0, 1)^\top$, supposed $\Omega$ is rotated in the above way.
\end{theorem}

\begin{proof}
We assume that $\Omega$ is rotated as stated in the theorem. Let $u \in \cH_\Gamma$ be a minimizer of \eqref{eq:minMaxEv1OneForm3}. Then $u$ satisfies the boundary conditions
\begin{equation}
\label{eq:uHgamma}
\begin{cases}
 \nu_1 u_1 + \nu_2 u_2 = 0 \quad \text{almost everywhere on}~\Gamma; \\
 \tau_1 u_1 + \tau_2 u_2 = 0 \quad \text{almost everywhere on}~\Gamma^c.
\end{cases}
\end{equation}
Define
\begin{align*}
 v = \binom{v_1}{v_2} := \binom{|u_1|}{|u_2|}.
\end{align*}
Our assumptions on the normal directions of $\Gamma$ and $\Gamma^c$ imply
\begin{equation}
\label{eq:NormalsQuadr}
\begin{cases}
 \nu_1 \nu_2 \le 0 \quad \text{almost everywhere on}~\Gamma, \\
 \tau_1 \tau_2 \le 0 \quad \text{almost everywhere on}~\Gamma^c, \\
\end{cases}
\end{equation}
for the components of the normal and tangential unit vector fields. From this and \eqref{eq:uHgamma} it follows immediately that $v$ satisfies the same boundary conditions as $u$, i.e.\ $v \in \cH_\Gamma$. Moreover, due to
\begin{equation*}
 \partial_j v_j = \partial_j \abs{u_j} = \sign(u_j) \abs{\partial_j u_j} \le |\partial_j u_j|
\end{equation*}
the quotient
\begin{align*}
 \frac{\int_\Omega \big( |\nabla v_1|^2 + |\nabla v_2|^2 \big) - \int_{\partial \Omega} \kappa \big( |v_1|^2 + |v_2|^2 \big)}{\int_\Omega \left( |v_1|^2 + |v_2|^2 \right)}
\end{align*}
is less or equal the same term with $v$ replaced by $u$; however, as $u$ is a minimizer of the Rayleigh quotient in \eqref{eq:minMaxEv1OneForm3}, it follows that $v$ is a minimizer as well. Then by Theorem \ref{thm:VariationalPrincipleOneForm3}, $A v = \eta_1 v$, that is,
\begin{equation*}
 - \Delta v_j = \eta_1 v_j \geq 0, \quad j = 1, 2,
\end{equation*}
cf.\ Remark \ref{rem:action}. Hence, by the maximum principle for subharmonic functions applied to $- v_j$, each $v_j$ is either non-zero everywhere in $\Omega$ or constantly equal to zero. In particular, $u_j$ is either non-zero everywhere in $\Omega$ or identically zero, $j = 1, 2$. 

Next we argue that both components of $u$ must be strictly positive in $\Omega$. Assume for a contradiction that, without loss of generality, $u_1 = 0$ constantly in $\Omega$. We prove that this implies that $\Omega$ is an axioparallel rectangle, which contradicts the assumptions on the angles between $\Gamma$ and $\Gamma^c$. In fact, if $\Gamma^c$ contains a part which is not axioparallel, then there exists a non-empty, relatively open set $\Sigma \subset \Gamma^c$ such that $\tau_2 \neq 0$ almost everywhere on $\Sigma$ and, hence, $u |_{\Gamma^c} \cdot \tau = 0$ implies $u_2 |_\Sigma = 0$; on the other hand, by \eqref{eq:eigenspace}, $u$ can be written $u = \nabla \psi + \nabla^\perp \phi$ for some $\psi \in \ker (- \Delta_{\Gamma_c} - \eta_1)$ and $\phi \in \ker (- \Delta_{\Gamma} - \eta_1)$; hence, $\diver u = \Delta \psi = - \eta_1 \psi$ and, thus, $0 = \diver u |_{\Gamma^c} = \partial_2 u_2 |_\Sigma$. Since $\ee_2$ is non-tangential on $\Sigma$, we get $\partial_\nu u_2 |_\Sigma = 0$. As $- \Delta u_2 = \eta_1 u_2$ and $u_2$ together with its normal derivative vanishes on $\Sigma$, a unique continuation argument, see, e.g., \cite[Lemma 2.2]{LR17}, gives $u_2 = 0$ constantly in $\Omega$, a contradiction. Analogously, we can use the boundary conditions $u |_\Gamma \cdot \nu = 0$ and $\omega (u) |_\Gamma = 0$ to conclude that $\Gamma$ consists of axioparallel pieces only. Thus $\Omega$ is a rectangle, a contradiction to the angle assumption in Hypothesis \ref{hyp2}. Hence we have shown that both $u_1$ and $u_2$ are strictly non-zero in $\Omega$, and we can conclude further from \eqref{eq:uHgamma} and \eqref{eq:NormalsQuadr} that $\sgn (u_1) = \sgn (u_2)$.

Next, let us show that
\begin{align}\label{eq:crucialIneq}
 \eta_1 = \lambda_1^{\Gamma^c} < \lambda_1^\Gamma
\end{align}
holds. For a contradiction, assume $\eta_1 = \lambda_1^\Gamma$ and let $\phi \in \ker (- \Delta_\Gamma - \eta_1)$ be non-trivial. Then by Proposition \ref{prop:TranslateEV}, $u = \nabla^\perp \phi \in \cH_\Gamma$ is a minimizer of \eqref{eq:minMaxEv1OneForm3}. Thus, according to what we have just seen, we may choose the sign of $\phi$ such that both $u_1 = - \partial_2 \phi$ and $u_2 = \partial_1 \phi$ are strictly positive in $\Omega$; that is, $\phi$ is strictly increasing in the directions of both $\ee_1$ and $- \ee_2$. By the assumptions on $\Omega$, there exist points $x_1, x_2 \in \Gamma$ and a path through $\Omega$ from $x_1$ to $x_2$ consisting of axio-parallel line segments each of which points into the direction of either $\ee_1$ or $- \ee_2$; cf.\ Figure~\ref{fig:DomainWithPaths}.
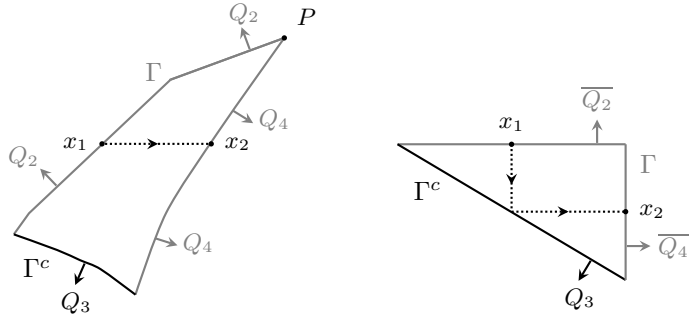
\begin{figure}[h]
\begin{tikzpicture}[scale=0.4]
\pgfsetlinewidth{0.8pt}
\node[white] (A) at (2,4) {};
\node[white] (B) at (6,2) {};
\node[white] (C) at (10.89392,10.5121) {};
\node[white] (D) at (7.15202,9.12554) {};
\node[white] (E) at (3.42633,3.48028) {};
\node[white] (F) at (4.17903,3.14355) {};
\node[white] (G) at (4.8525,2.84643) {};
\node[white] (H) at (7.011,4.71739) {};
\node[white] (I) at (8.74977,7.44076) {};
\draw(2.76,3.5) node[black][left, below]{$\gac$};
\draw(6.6,9.4) node[gray]{$\ga$};
\draw[gray] (C.center) -- node[sloped,inner sep=0cm,above,pos=0.3, anchor=south west,
minimum height=0.3cm,minimum width=1cm](M){}(7.15,9.13) -- (D.center);
\draw[gray] (D.center) -- node[sloped,inner sep=0cm,above,pos=0.8, anchor=south west,
minimum height=0.3cm,minimum width=1cm](N){}(2.5,4.7) -- (A.center);
\path[gray] (N.south west) edge[-stealth,shorten <=0.5pt] node[above left,pos=0.5] {\small $Q_2$} (N.north west);
\path[gray] (M.south west) edge[-stealth,shorten <=0.5pt] node[above,pos=0.7] {\small $Q_2$} (M.north west);
\draw[white] (F.center) -- node[sloped,inner sep=0cm,above,pos=.5, anchor=north east, minimum height=0.3cm,minimum width=1cm](F1){}(4.51,3) -- (G.center);
\path (F1.north east) edge[-stealth,shorten <=0.5pt] node[below,pos=0.9] {\small $Q_3$} (F1.south east);
\draw[white] (B.center) -- node[sloped,inner sep=0cm,above,pos=.5, anchor=north east, minimum height=0.3cm,minimum width=1cm](F2){}(7.2,5.77) -- (H.center);
\path[gray] (F2.north east) edge[-stealth,shorten <=0.5pt] node[right,pos=0.8] {\small $Q_4$} (F2.south east);
\draw[white] (I.center) -- node[sloped,inner sep=0cm,above,pos=.5, anchor=north east, minimum height=0.3cm,minimum width=1cm](F3){}(9.6,8.8) -- (C.center);
\path[gray] (F3.north east) edge[-stealth,shorten <=0.5pt] node[right,pos=0.8] {\small $Q_4$} (F3.south east);
\draw[gray] plot[smooth] coordinates {(B) (H) (I) (C)};
\draw plot[smooth] coordinates {(A) (E) (F) (G) (B)};
\draw[gray] (C.center) -- (D.center);
\node[circle,fill=black,inner sep=0pt,minimum size=2pt,label=above right:{\small {$P$}}] (C1) at (10.89392,10.5121){};
\begin{scope}[densely dotted,decoration={
    markings,
    mark=at position 0.5 with {\arrow{stealth}}}
    ] 
    \draw[postaction={decorate}] (5,7) to (8.49,7);
\end{scope}
\node[circle,fill=black,inner sep=0pt,minimum size=2pt,label=right:{\small {$x_2$}}] (X2) at (8.49,7){};
\node[circle,fill=black,inner sep=0pt,minimum size=2pt,label=left:{\small {$x_1$}}] (X11) at (4.9,7){};
\end{tikzpicture}
\hspace{0.6cm}
\begin{tikzpicture}[scale=0.6]
\pgfsetlinewidth{0.8pt}
\node[white] (A) at (0,3) {};
\node[white] (B) at (5,3) {};
\node[white] (C) at (5,0) {};
\draw(0.66,1.97) node{$\gac$};
\draw[gray](A.center) -- node[sloped,inner sep=0cm,above,pos=3.5, anchor=south east, minimum height=0.3cm,minimum width=1cm](F1){}(1.25,3) -- (B.center);
\path[gray] (F1.south east) edge[-stealth,shorten <=0.5pt] node[above,pos=0.9] {\small $\overline{Q_2}$} (F1.north east);
\draw[gray](5.5,2.57) node{$\ga$};
\draw(C.center) -- node[sloped,inner sep=0cm,above,pos=0.3, anchor=north east, minimum height=0.3cm,minimum width=1cm](F2){}(2.5,1.5) --(A.center);
\path(F2.north east) edge[-stealth,shorten <=0.5pt] node[below,pos=0.9] {\small $Q_3$} (F2.south east);
\draw[gray] (B.center) -- node[sloped,inner sep=0cm,above,pos=1.5, anchor=south west, minimum height=0.3cm,minimum width=1cm](F3){}(5,1.5) --(C.center);
\path[gray] (F3.south west) edge[-stealth,shorten <=0.5pt] node[right,pos=0.9] {\small $\overline{Q_4}$} (F3.north west);
\begin{scope}[densely dotted,decoration={
    markings,
    mark=at position 0.6 with {\arrow{stealth}}}
    ] 
    \draw[postaction={decorate}] (2.5,3) to (2.5,1.51);
\end{scope}
\begin{scope}[densely dotted,decoration={
    markings,
    mark=at position 0.5 with {\arrow{stealth}}}
    ] 
    \draw[postaction={decorate}] (2.5,1.51) to (5,1.51);
\end{scope}
\node[circle,fill=black,inner sep=0pt,minimum size=2pt,label=right:{\small {$x_2$}}] (X2) at (5,1.51){};
\node[circle,fill=black,inner sep=0pt,minimum size=2pt,label=above:{\small {$x_1$}}] (X1) at (2.5,3){};
\end{tikzpicture}
\caption{The eigenfunction $\phi$ would simultaneously be vanishing on $\ga$ and strictly increasing along both pictured paths.}
\label{fig:DomainWithPaths}
\end{figure}
In particular, $\phi$ is strictly increasing along this path. On the other hand, $\phi (x_1) = 0 = \phi (x_2)$, a contradiction. Thus we have proven \eqref{eq:crucialIneq}. In particular, we have shown the inequality \eqref{eq:mainEVin}.

To summarize, we have shown that $u \in \cH_\Gamma$ is a minimizer of \eqref{eq:minMaxEv1OneForm3} if and only if $u = \nabla \psi$ for some non-trivial $\psi \in \ker (- \Delta_{\Gamma^c} - \lambda_1^{\Gamma^c})$, and that in this case both components of $\nabla \psi$ are strictly positive. Hence, the eigenfunction $\psi$ is strictly increasing in $\Omega$ in the directions of both $\ee_1$ and $\ee_2$. This completes the proof.
\end{proof}

Eigenvalue inequalities of the form \eqref{eq:mainEVin} for Laplacians with mixed boundary conditions were obtained earlier by the authors of this article in \cite{AR23} using classical variational principles. Theorem~3.1 in \cite{AR23} requires the stronger assumptions that $\Omega$ is convex and $\Gamma^c$ is a straight line segment. On the other hand, it does not require $\Gamma$ to contain at least one corner, while the assumption on the exterior unit normals of $\Gamma$ in Theorem \ref{thm:MainResult} forces $\Gamma$ to contain one corner with interior angle at most $\pi/2$ at which the exterior normal changes between the second and fourth quadrant.

The next example shows that the inequality \eqref{eq:mainEVin} no longer needs to hold once the interior angles at the end points of $\Gamma$ are equal to $\pi/2$.

\begin{example}
Let $\Omega = (0, \pi) \times (0, \pi)$ and let $\Gamma$ consist of the sides $y = \pi$ and $x = \pi$ while $\Gamma^c$ consists of the sides $y = 0$ and $x = 0$. Then the interior angles of $\partial \Omega$ at the end points of $\Gamma$ equal $\pi/2$ while all further assumptions of Theorem \ref{thm:MainResult} are satisfied. However, the eigenvalue inequality \eqref{eq:mainEVin} fails since $\lambda_1^{\Gamma^c} = \frac{1}{2} = \lambda_1^\Gamma$ as one computes easily by separating variables.
\end{example}

The following corollary concerns the hot spots property for the first eigenfunction of the operator $-\Delta_{\gac}$. 

\begin{corollary}
\label{cor:MonotonicityEF}
Assume that Hypothesis \ref{hyp2} is satisfied. In addition, assume that $\nu(x) \in \overline{Q_3}$ for almost all $x \in \gac$ and $\nu(x) \in \overline{Q_2} \cup \overline{Q_4}$ for almost all $x \in \ga$. Then the eigenfunction of $-\Delta_{\gac}$ corresponding to its first eigenvalue $\lambda_1^{\Gamma^c}$, chosen positive in $\Omega$, takes its maximum only on $\Gamma$. If 
$\Gamma$ does not contain any axioparallel segment, then the point in $\overline \Omega$ at which the eigenfunction takes its maximum is unique and coincides with the corner $P$ at which the normal vector makes a jump between the quadrants $Q_2$ and $Q_4$.
\end{corollary}

\begin{proof}
By Theorem \ref{thm:MainResult} the eigenfunction is strictly increasing in $\Omega$ in the directions of $\ee_1$ and $\ee_2$. In particular, as for each point $x \in \Omega$ the halfline $\{x + t \ee_1 : t > 0\}$ intersects $\Gamma$, it follows that the maximum of the eigenfunction is attained only in $\Gamma$. If $\Gamma$ does not contain any axioparallel line segment, then $\Gamma$ contains points $x_1, x_2$ such that $\nu (x_1) \in Q_2$ and $\nu (x_2) \in Q_4$, and there cannot exist more than one corner at which $\nu$ makes a jump between $Q_2$ and $Q_4$; otherwise $\Gamma$ needs to contain an inward-pointing corner. If $P$ denotes the unique corner at which $\nu$ jumps from $Q_2$ to $Q_4$, for any $x \in \Gamma$, $x \neq P$, there exists a path through $\Omega$ from $x$ to $P$ consisting of axioparallel line segments pointing into the directions of either $\ee_1$ or $\ee_2$, and the eigenfunction is strictly increasing along this path. Hence its unique maximum is attained at $P$.
\end{proof}

According to Corollary \ref{cor:MonotonicityEF}, in the case of the first domain in Figure \ref{fig:DomainWithPaths} the eigenfunction corresponding to $\lambda_1^{\gac}$ takes its maximum at the point $P$ only.

Next, we apply Theorem \ref{thm:MainResult} and Corollary \ref{cor:MonotonicityEF} to the recently much studied case of triangles.

\begin{example}
Assume that $\Omega$ is an acute triangle, $\Gamma^c$ is one of its sides and $\Gamma$ consists of the remaining two sides. Further, assume that $\Omega$ is rotated such that $\Gamma$ and $\Gamma^c$ satisfy the condition on the normal vectors required in Theorem \ref{thm:MainResult} and such that $\Gamma$ contains no axioparallel part; this is always possible as long as the triangle is acute, see Figure~\ref{fig:AcuteTriangle}.
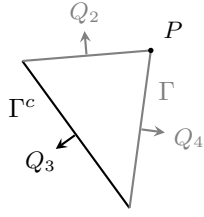
\begin{figure}[h]
\begin{tikzpicture}
\pgfsetlinewidth{0.8pt}
\node[white] (A) at (1,3) {};
\node[white] (B) at (2.41,1.05) {};
\node[white] (C) at (2.69,3.14) {};
\draw(1,2.38) node{$\gac$};
\draw(A.center) -- node[sloped,inner sep=0cm,above,pos=1, anchor=north east, minimum height=0.3cm,minimum width=1cm](F1){}(1.705,2.025) -- (B.center);
\path (F1.north east) edge[-stealth,shorten <=0.5pt] node[below left,pos=0.5] {\small $Q_3$} (F1.south east);
\draw[gray](2.9,2.6) node{$\ga$};
\draw[gray] (C.center) -- node[sloped,inner sep=0cm,above,pos=1, anchor=south east, minimum height=0.3cm,minimum width=1cm](F2){}(1.845,3.07) --(A.center);
\path[gray] (F2.south east) edge[-stealth,shorten <=0.5pt] node[above,pos=0.9] {\small $Q_2$} (F2.north east);
\draw[gray] (B.center) -- node[sloped,inner sep=0cm,above,pos=1, anchor=north east, minimum height=0.3cm,minimum width=1cm](F3){}(2.55,2.095) --(C.center);
\path[gray] (F3.north east) edge[-stealth,shorten <=0.5pt] node[left,pos=3] {\small $Q_4$} (F3.south east);
\node[circle,fill=black,inner sep=0pt,minimum size=2pt,label=above right:{\small {$P$}}] (C) at (2.69,3.14){};
\end{tikzpicture}
\caption{The first eigenfunction of $-\Delta_{\gac}$ takes its maximum at $P$.}
\label{fig:AcuteTriangle}
\end{figure}
Then the first eigenfunction $\psi_1$ of $- \Delta_{\Gamma^c}$ is strictly increasing in the directions of $\ee_1$ and $\ee_2$, implying that $\psi_1$, chosen positive, takes its maximum at a unique point, namely the corner inside the Neumann part $\Gamma$; cf.\ Corollary \ref{cor:MonotonicityEF}. This observation is slightly more general than Theorem 1.3 in the recent preprint \cite{LY24p} and Theorem 1.1 of \cite{H24p} when applied to the acute case. The statement on the only maximum lying at the Neumann vertex confirms Corollary 1.6 (3) of \cite{LY24p}.
\end{example}

\appendix

\section{A Helmholtz type decomposition}
\label{sec:appendixHH}

In this appendix we provide an elementary proof for a Helmholtz type orthogonal decomposition of the space $L^2 (\Omega; \C^2)$ which is suitable for mixed boundary conditions. A proof in the language of differential forms can be found in \cite{BPS19}. Recall the notation 
\begin{align*}
 H_{0, \Sigma}^1 (\Omega) = \left\{ \psi \in H^1 (\Omega) : u |_\Sigma = 0 \right\}
\end{align*}
for any non-empty, relatively open set $\Sigma \subset \partial \Omega$. In the proof of the following theorem we will make use of the Poincar\'e inequality
\begin{align}\label{eq:Poincare}
 \int_\Omega |\nabla \psi|^2 \geq c \int_\Omega |\psi|^2, \quad \psi \in H_{0, \Sigma}^1 (\Omega),
\end{align}
where the optimal constant $c$ equals the eigenvalue $\lambda_1^\Sigma$.

\begin{theorem}
\label{thm:HelmholtzDecGeneral}
Assume that Hypothesis \ref{hyp1} is satisfied. Then the orthogonal decomposition 
\begin{equation}
\label{eq:HelmholtzDecGeneral}
 L^2(\Omega; \mathbb{C}^2) = \nabla \hmixc \oplus \nabla^{\perp} \hmix \oplus H_c
\end{equation}
holds, where 
\begin{equation*}
 H_c = \left\{ u \in L^2 (\Omega; \C^2) : \diver{u} = \omega(u) = 0,  u |_{\Gamma} \cdot \nu = 0, u |_{\Gamma^c} \cdot \tau = 0 \right\}.
\end{equation*}
\end{theorem}

\begin{proof}
Let us first show that each of the spaces on the right-hand side of \eqref{eq:HelmholtzDecGeneral} is a closed subspace of $L^2 (\Omega; \C^2)$. We first prove that $\nabla \hmixc$ is closed. Let $\psi_n \in H_{0, \Gamma^c}^1 (\Omega)$ such that $(\nabla \psi_n)_n$ is a Cauchy sequence in $L^2 (\Omega; \C^2)$, that is,
\begin{equation*}
 \int_{\om} \abs{\nabla \psi_n - \nabla \psi_m}^2 \to 0
\end{equation*}
as $m, n \to +\infty$. Then there exists an element $u \in \ldd$ such that $\nabla \psi_n \to u$ in $\ldd$. Moreover, by the Poincar\'e inequality \eqref{eq:Poincare},
\begin{equation*}
 \int_{\om} \abs{\psi_n - \psi_m}^2 \leq \frac{1}{c} \int_{\om} \abs{\nabla \psi_n - \nabla \psi_m}^2 \to 0.
\end{equation*}
Thus there exists $\psi \in L^2 (\Omega)$ such that $\psi_n \to \psi$ in $L^2 (\Omega)$. Furthermore, for each $\phi \in C_0^\infty (\Omega)$ we have
\begin{align*}
 \int_\Omega \psi \partial_j \phi = \lim_{n \to \infty} \int_\Omega \psi_n \partial_j \phi = - \lim_{n \to \infty} \int_\Omega \partial_j \psi_n \phi = - \int_\Omega u_j \phi, \quad j = 1, 2,
\end{align*}
implying $\psi \in H^1 (\Omega)$ and $\nabla \psi = u$ in $L^2 (\Omega; \C^2)$. We have therefore $\psi_n \to \psi$ in $H^1 (\Omega)$, and as $H_{0, \Gamma^c}^1 (\Omega)$ is complete, it follows $\psi \in H_{0, \Gamma^c}^1 (\Omega)$. Hence $\nabla H_{0, \Gamma^c}^1 (\Omega)$ is complete. The proof of the completeness of $\nabla^\perp H_{0, \Gamma}^1 (\Omega)$ is analogous.

Now consider $u^n \in H_c$ such that $(u^n)_n$ is a Cauchy sequence in $L^2 (\Omega; \C^2)$. Since $\diver u^n = \omega (u^n) = 0$ for all $n$, $(u^n)_n$ is a Cauchy sequence in both $H (\diver; \Omega)$ and $H (\omega; \Omega)$. Hence there exists some $u \in H (\diver; \Omega) \cap H (\omega; \Omega)$ such that $(u^n)_n$ converges to $u$ in both $H (\diver; \Omega)$ and $H (\omega; \Omega)$. In particular, by continuity, $\diver{u}=\omega(u)=0$ follows. Since the tangent and normal traces are continuous mappings from $H (\omega; \Omega)$ or $H (\diver; \Omega)$, respectively, into $H^{-1/2}(\partial \Omega)$, the identities $u |_\Gamma \cdot \nu = 0$ and $u |_{\Gamma^c} \cdot \tau = 0$ follow. Hence $H_c$ is complete.

Let us prove next that the spaces on the right-hand side of \eqref{eq:HelmholtzDecGeneral} are mutually orthogonal. For this let first $\psi \in H_{0, \Gamma^c}^1 (\Omega)$ and $\phi \in H_{0, \Gamma}^1 (\Omega)$. Integration by parts \eqref{eq:PI} gives
\begin{align*}
 \int_{\om} u \cdot v = \int_{\om} \nabla \psi \cdot \nabla^{\perp} \phi = - \int_{\om} \psi \, \overline{\diver \nabla^\perp \phi} - \left( \psi |_{\bnd}, \nabla \phi |_{\partial \Omega} \cdot \tau \right)_{\bnd} = 0
\end{align*}
due to the boundary conditions of $\psi$ and $\phi$, and $\diver \nabla^\perp \phi= 0$. Now let in addition $u \in H_c$. Then
\begin{align*}
 \int_{\om} u \cdot \left( \nabla \psi + \nabla^\perp \phi \right) & = - \int_{\om} \diver{u} \, \overline{\psi} + \left(u |_{\bnd} \cdot \nu, \psi |_{\partial \Omega} \right)_{\bnd} \\
 & \quad - \int_{\om} \omega (u) \, \overline{\phi} + \left(u |_{\bnd} \cdot \tau, \phi |_{\partial \Omega} \right)_{\bnd} = 0.
\end{align*}
Thus $\nabla H_{0, \Gamma^c}^1 (\Omega)$ and $\nabla^\perp H_{0, \Gamma}^1 (\Omega)$ are orthogonal to each other and both spaces are orthogonal to $H_c$.

To establish \eqref{eq:HelmholtzDecGeneral} it remains to show that each $u \in L^2 (\Omega; \C^2)$ which is orthogonal to $\nabla \hmixc \oplus \nabla^\perp \hmix$ belongs to $H_c$. Indeed, for such $u$ we have 
\begin{align*}
 \int_\Omega u \cdot \nabla \psi = 0 = \int_\Omega u \cdot \nabla^\perp \phi
\end{align*}
for all $\psi, \phi \in C_0^\infty (\Omega)$, implying $\diver u = 0 = \omega (u)$. In particular, $u |_{\partial \Omega} \cdot \nu, u |_{\partial \Omega} \cdot \tau \in H^{- 1/2} (\partial \Omega)$ are well-defined. Furthermore, for each $\psi \in H_{0, \Gamma^c}^1 (\Omega)$,
\begin{align*}
 0 = \int_\Omega u \cdot \nabla \psi = \left(u |_{\partial \Omega} \cdot \nu, \psi |_{\partial \Omega} \right)_{\partial \Omega},
\end{align*}
giving $u |_\Gamma \cdot \nu = 0$. Analogously,
\begin{align*}
 0 = \int_\Omega u \cdot \nabla^\perp \phi = \left(u |_{\partial \Omega} \cdot \tau, \phi |_{\partial \Omega} \right)_{\partial \Omega}
\end{align*}
holds for all $\phi \in H_{0, \Gamma}^1 (\Omega)$, giving $u |_{\Gamma^c} \cdot \tau = 0$. Thus we have shown $u \in H_c$. This completes the proof.
\end{proof}

We conclude this appendix with an observation on the space $H_c$. In Example \ref{ex:square} we have seen that $H_c$ may be non-trivial if $\ga$ is not connected. In fact, the ideas of the proof of Proposition \ref{prop:ThirdSpaceTrivial} may be used to estimate the dimension of $H_c$ and, hence, the dimension of the kernel of the operator $A$ in Section \ref{sec:sec3}; this is exemplified in the following proposition.

\begin{proposition}
\label{prop:dimHC}
Let $\Omega \subset \R^2$ be a bounded, simply connected Lipschitz domain whose boundary consists of finitely many $C^\infty$-smooth arcs and has no inward-pointing corners. Let $\Gamma, \Gamma^c \subset \bnd$ be non-empty, relatively open subsets of the boundary such that $\Gamma \cap \Gamma^c = \emptyset$ and $\overline{\Gamma \cup \Gamma^c} = \partial \Omega$, and such that the interior angles at all points where $\Gamma$ and $\Gamma^c$ meet are strictly below $\pi/2$. Assume that $\ga$ consists of two connected components. Then 
\begin{equation*}
\dim{H_c} \le 1.
\end{equation*}
\end{proposition}

\begin{proof}
Let us assume for a contradiction that there exist $u,v \in H_c$ such that $\int_\Omega |u|^2 = 1 = \int_\Omega |v|^2$ and $\int_\Omega u \cdot v = 0$. Under the assumptions of the proposition we can argue as in the proof of Lemma \ref{lem:formDomain} to show that $u,v \in H^1(\Omega; \C^2)$ and thus $u|_{\bnd} \cdot \nu, u|_{\bnd} \cdot \tau, v|_{\bnd} \cdot \nu, v|_{\bnd} \cdot \tau \in L^2(\bnd)$. Since $\om$ is simply connected, $\omega(u) = 0$ in $\Omega$ implies that there exists $f \in H^2 (\Omega)$ such that $u = \nabla f$, see \cite[Chapter I, Theorem 2.9]{GR}, and similarly $\omega(u^\perp) = \diver{u} = 0$ in $\Omega$ implies that there exists $g \in H^2 (\Omega)$ such that $u^\perp = \nabla g$. Furthermore, we have
\begin{align}\label{eq:BCfg}
 \partial_\tau f |_{\gac} = u |_{\Gamma^c} \cdot \tau = 0, \quad \text{and} \quad \partial_\tau g |_{\ga} = u |_\Gamma \cdot \nu = 0.
\end{align}
Let us denote by $\Gamma_1$ and $\Gamma_3$ the connected components of $\Gamma$. Since $\om$ is simply connected, $\gac$ also consists of two connected components $\Gamma_2$ and $\Gamma_4$, where the enumeration is chosen following the boundary in positive direction, see Figure~\ref{fig:gamma1234}. Let us call the end points of $\Gamma_2$ ordered following positive orientation $P$ and $Q$. 
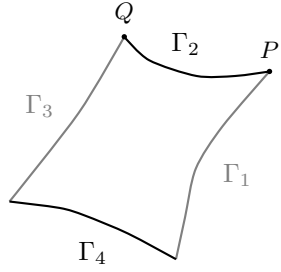
\begin{figure}[h]
\begin{tikzpicture}[scale=0.5]
\pgfsetlinewidth{0.8pt}
\node[white] (A) at (2,3) {};
\node[white] (B) at (6.38,1.47) {};
\node[white] (C) at (3.32,2.83) {};
\node[white] (D) at (3.98,2.63) {};
\node[white] (E) at (5.12,2.15) {};
\node[white] (F) at (5.02,7.37) {};
\node[white] (G) at (3.82,5.35) {};
\node[white] (H) at (8.84,6.45) {};
\node[white] (L) at (6.92,6.33) {};
\node[white] (M) at (7.94,6.33) {};
\node[white] (Q) at (5.67,6.75) {};
\node[white] (S) at (6.62,2.59) {};
\node[white] (T) at (6.9,3.87) {};
\node[white] (U) at (7.6,4.99) {};
\draw[gray] plot[smooth] coordinates {(B) (S) (T) (U) (H)};
\draw (7.36,3.7) node[gray][right]{$\Gamma_1$};
\draw plot[smooth] coordinates { (F) (Q) (L) (M) (H)};
\draw (6.64,6.62) node[black][above]{$\Gamma_2$};
\draw[gray] plot[smooth] coordinates {(F) (G) (A)};
\draw (3.44,5.54) node[gray][left]{$\Gamma_3$};
\draw plot[smooth] coordinates {(A) (C) (D) (E) (B)};
\draw (4.19,2.2) node[black][below]{$\Gamma_4$};
\node[circle,fill=black,inner sep=0pt,minimum size=2pt,label=above:{\small {$Q$}}] (F1) at (5.02,7.37){};
\node[circle,fill=black,inner sep=0pt,minimum size=2pt,label=above:{\small {$P$}}] (H1) at (8.84,6.45){};
\end{tikzpicture}
\caption{Each $\Gamma_i$ is contained in a level set of the potentials of elements of $H_c$.}
\label{fig:gamma1234}
\end{figure}
By \eqref{eq:BCfg}, $f$ is constant along $\Gamma_2$ and along $\Gamma_4$, and we can assume without loss of generality that $f|_{\ga_2} = a_2$ for some constant $a_2 \in \mathbb{C}$ and $f |_{\Gamma_4} = 0$. Analogously, we may assume that $g|_{\ga_1} = a_1$ for some $a_1 \in \C$ and $g|_{\ga_3} = 0$. The same observations hold for $v$: there exist $p,q \in H^2 (\Omega)$ such that $v = \nabla p$, $v^\perp = \nabla q$, $q|_{\ga_1} = b_1$ and $p|_{\ga_2} = b_2$ identically for some $b_1, b_2 \in \mathbb{C}$, $q|_{\ga_3} = 0$ and $p|_{\ga_4} = 0$. We compute
\begin{align}
\label{eq:compUV}
\begin{split}
 0 & = \int_{\om} u \cdot v = \int_{\om} \nabla f \cdot v = - \int_{\om} f \, \overline{\diver{v}} + \int_{\bnd} f \, \nu \cdot v = \int_{\gac}  f \, \overline{\partial_\tau q} \\
 & = a_2\int_{\ga_2} \overline{\partial_\tau q} = a_2 (\overline{q}(Q)-\overline{q}(P)) = - a_2 \overline{b_1},
\end{split}
\end{align}
where we used the fundamental theorem of calculus. Thus $a_2 = 0$ or $b_1 = 0$. Analogous computations yield
\begin{equation*}
 \int_{\om} |u|^2 = \int_{\om} \nabla f \cdot u = - a_2 \overline{a_1}
\end{equation*}
and 
\begin{equation*}
 \int_{\om} |v|^2 = \int_{\om} v \cdot \nabla p = - b_2 \overline{b_1},
\end{equation*}
and it follows $\int_\Omega |u|^2 = 0$ or $\int_\Omega |v|^2 = 0$, a contradiction.
\end{proof}

\section*{Acknowledgements}
The authors gratefully acknowledge financial support by the grants no.\ 2018-04560 and 2022-03342 of the Swedish Research Council (VR).

\end{document}